\documentclass[article,hidelinks,onefignum,onetabnum]{siamart220329}

\def\bfg{{\bf g}}
\def\bfd{{\bf d}}
\def\Prox{{\mbox{Prox}}}

\usepackage{lipsum}
\usepackage{amsfonts}
\usepackage{graphicx}
\usepackage{epstopdf}
\usepackage{algorithmic}
\ifpdf
  \DeclareGraphicsExtensions{.eps,.pdf,.png,.jpg}
\else
  \DeclareGraphicsExtensions{.eps}
\fi

\newsiamremark{remark}{Remark}
\newsiamremark{hypothesis}{Hypothesis}
\crefname{hypothesis}{Hypothesis}{Hypotheses}
\newsiamthm{claim}{Claim}

\headers{Convergence of ZH-type nonmonotone descent method}{Y. T. Qian, T. Tao, S. H. Pan and H. D. Qi}

\title{Convergence of ZH-type nonmonotone descent method for Kurdyka-{\L}ojasiewicz optimization problems\thanks{Submitted to the editors DATE.
\funding{The second author's work was funded by the Natural Science Foundation of Guangdong Province under project 2023A1515111167. The third author's work was funded by the National Natural Science Foundation of China under project 12371299. 
Qian and Qi were funded by Hong Kong RGC General Research Fund PolyU/15309223 and PolyU AMA Project 230413007.}}}

\author{Yitian Qian\thanks{Department of Applied Mathematics, The Hong Kong Polytechnic University, Hong Kong 
  (\email{yitian.qian@polyu.edu.hk}, \email{houduo.qi@polyu.edu.hk}).}
\and Ting Tao\thanks{School of Mathematics, Foshan University, Foshan 
  (\email{taoting@fosu.edu.cn}).}
\and Shaohua Pan\thanks{School of Mathematics, South China University of Technology, Guangzhou
  (\email{shhpan@scut.edu.cn}).}
\and Houduo Qi\footnotemark[2]}

\usepackage{amsopn}

\ifpdf
\hypersetup{
  pdftitle={Convergence of ZH-type nonmonotone descent method for Kurdyka-{\L}ojasiewicz optimization problems},
  pdfauthor={Y. T. Qian, T. Tao, S. H. Pan and H. D. Qi}
}
\fi

\newtheorem{aassumption}{Assumption}

\begin{document}

\maketitle

\begin{abstract}
We propose a novel iterative framework for minimizing a proper lower semicontinuous Kurdyka-{\L}ojasiewicz (KL) function $\Phi$. It comprises a Zhang-Hager (ZH-type) nonmonotone decrease condition and a relative error condition. Hence, the sequence generated by the ZH-type nonmonotone descent methods will fall within this framework. Any sequence conforming to this framework is proved to converge to a critical point of $\Phi$. If in addition $\Phi$ has the KL property of exponent $\theta\!\in(0,1)$ at the critical point, the convergence has a linear rate for $\theta\in(0,1/2]$ and a sublinear rate of exponent $\frac{1-\theta}{1-2\theta}$ for $\theta\in(1/2,1)$. To the best of our knowledge, this is the first work to establish the full convergence of the iterate sequence generated by a ZH-type nonmonotone descent method for nonconvex and nonsmooth optimization problems. The obtained results are also applied to achieve the full convergence of the iterate sequences produced by the proximal gradient method and Riemannian gradient method with the ZH-type nonmonotone line-search.
\end{abstract}



\section{Introduction}\label{sec1}

For a descent method for nonconvex and nonsmooth optimization problems, 
it has become a common task to establish the full convergence of its iterate sequence under the KL property of objective functions. This task has been strongly supported and motivated by the work \cite{Absil05} for gradient-related methods, 
\cite{Attouch09} for proximal point algorithms, \cite{Attouch13} for forward-backward algorithms, \cite{Attouch10,Bolte14} for  proximal alternating minimization algorithms, and \cite{Xu13} for block coordinate descent algorithms, to just name a few of the representative papers in this research direction. The convergence analysis of all these works is based on the (sufficiently) monotone decrease of objective value sequences. Although the (sufficiently) monotone decrease is crucial to achieve the full convergence, it also leads to relatively short step sizes. 
To overcome this drawback, nonmonotone variants of the sufficient decrease are often considered. In the context of line-search methods, Grippo et al. \cite{Grippo86} first proposed a nonmonotone line-search procedure by monitoring the maximum objective value attained by the latest iterates
(referred to as GLL-type in this paper); Zhang and Hager \cite{ZhangH04} later proposed a different strategy of nonmonotone line-search (referred to as the ZH-type in this paper) by taking a weighted average of the current objective value and those of the past iterates. These two types of nonmonotone line-search procedures are received as the most popular choices although they can be cast into the general framework proposed in \cite{Sachs11}. In the past decade, they have also been aligned with algorithms for nonconvex and nonsmooth optimization problems, 
such as composite optimization problems of minimizing the sum of a smooth function and a nonsmooth one \cite{Kanzow22,Marchi23,Li15,Themelis18}, Riemannian manifold optimization problems \cite{Wen13,Oviedo22}, and DC programs \cite{Ferreira21,Lu19}.   

On the one hand, nonmonotone line-search procedures bring well-recognized 
numerical benefits such as increasing the possibility of seeking better critical points, and potentially accelerating the convergence of algorithms \cite{Toint96}. 
On the other hand, the nonmonotone decrease of objective value sequences brings a great challenge for the full convergence analysis of the iterate sequence of the concerned algorithms. Recently, Qian and Pan \cite{QianPan23} investigated the full convergence of the iterate sequence conforming to an iterative framework proposed by the GLL-type decreasing condition for KL optimization problems, and achieved its full convergence under a condition that is shown to be sufficient and necessary if objective functions are weakly convex on a neighborhood of the set of critical points. We notice that the ZH-type nonmonotone line-search procedure has been widely applied to nonconvex and nonsmooth optimization problems (see, e.g., \cite{Marchi23,Wen13,Themelis18,Oviedo22,Li15}), but as far as we know, there is no work to study the full convergence of its iterate sequence even for smooth optimization problems \cite{ZhangH04,Grapiglia17,Sachs11}. This paper aims at resolving this open problem via a novel iterative framework. 
 
In the following, we first describe the novel procedure, then explain the main results of the paper. To this end, consider the nonconvex and nonsmooth problem 
 \begin{equation}\label{prob}
  \min_{x\in\mathbb{X}}\,\Phi(x),
 \end{equation}
 where $\mathbb{X}$ represents a finite dimensional real vector space endowed with the inner product $\langle \cdot,\cdot \rangle$ and its induced norm $\|\cdot\|$, and $\Phi\!:\mathbb{X}\to\overline{\mathbb{R}}:=(-\infty,\infty]$ is a proper lower semicontinuous (lsc) function that is bounded below on its domain ${\rm dom}\,\Phi$. We are interested in nonmonotone descent methods for \eqref{prob} to produce an iterate sequence $\{x^k\}_{k\in\mathbb{N}}\subset{\rm dom}\,\Phi$ complying with the following conditions: 
 \begin{itemize}
 \item [{\bf H1.}] For each $k\!\in\!\mathbb{N}$, $\Phi(x^{k})\!+\!a_{k}\|x^k\!-\!x^{k-1}\|^2\!\le\! C_{k-1}$, where $a_k\ge\underline{a}$ for some $\underline{a}\!>\!0$, and
 $C_{k}=(1\!-\!\tau_{k})C_{k\!-\!1}\!+\tau_{k}\Phi(x^{k})$ with $C_0=\Phi(x^0)$ and $\tau_{k}\!\in\![\tau,1]$ for a $\tau\!\in\!(0,1]$.

 \item[{\bf H2.}] There exists a nonnegative integer $k_1$ such that for each $k\ge k_1$, 
 \[
  {\rm dist}(0,\partial\Phi(x^{k}))\le\frac{1}{b_{k}}\sum_{i=k-k_1}^{k+k_1}\|x^{i}\!-\!x^{i-1}\|+\varepsilon_{k} 
  \]
  with $b_{k}>0$ and $\varepsilon_{k}\ge0$, where $\partial\Phi(x^{k})$ is the subdifferential of $\Phi$ at $x^{k}$. 

 \item[{\bf H3.}] There exists a convergent subsequence $\{x^{k_j}\}_{j\in\mathbb{N}}$ with $\lim_{j\to\infty}x^{k_j}=\overline{x}\in{\rm dom}\,\Phi$ such that $\limsup_{j\to\infty}\Phi(x^{k_j})\le\Phi(\overline{x})$.
 \end{itemize}
 Unless otherwise stated, the above $\{a_k\}_{k\in\mathbb{N}},\{b_k\}_{k\in\mathbb{N}}$ and $\{\varepsilon_k\}_{k\in\mathbb{N}}$ are assumed to satisfy 
 \begin{equation}\label{abk-cond}
  \sum_{k=1}^{\infty}b_k=\infty,\ 
  \overline{B} := \sup_{\mathbb{N}\ni k\ge k_1}\frac{1}{b_k}\sum_{i=k-k_1}^{k+k_1}\frac{1}{\sqrt{a_i}}<\infty\ \ {\rm and}\ \ \sum_{k=1}^{\infty}\varepsilon_k<\infty.
 \end{equation}
 Here, ``$:=$'' means ``define''. The recursion relation of $\{C_k\}_{k\in\mathbb{N}}$ in H1 comes from the ZH-type nonmonotone line-search procedure, and when $C_k$ takes $\Phi(x^k)$ or $\tau_k\equiv 1$, condition H1 degenerates to the sufficiently monotone descent condition. In condition H2, the introduction of a nonnegative integer $k_1$ aims to bound ${\rm dist}(0,\partial\Phi(x^{k}))$ from above by $2k_1\!+\!1$ successive iterate errors. As will be shown in Section \ref{sec4.2}, this plays a crucial role in the convergence analysis of nonmonotone Riemannian gradient descent methods. When $\Phi$ is a Lipschitz smooth function on $\mathbb{X}=\mathbb{R}^n$, the iterate sequence $\{x^k\}_{k\in\mathbb{N}}$ generated by the nonmonotone line-search algorithms
 in the style of ZH-type \cite{ZhangH04,Grapiglia17,Sachs11} under the direction assumption there falls within the above framework with $\varepsilon_k\equiv 0$.
We include a proof in Appendix.

The main contribution of this work is to resolve the aforementioned open problem by establishing the full convergence of any sequence $\{x^k\}_{k\in\mathbb{N}}$ conforming to conditions H1-H3 under the KL property of $\Phi$, and its R-linear convergence rate under the KL property of $\Phi$ with exponent $1/2$. The difficulty to achieve this goal is how to divide an iterate sequence $\{x^k\}_{k\in\mathbb{N}}$ obeying conditions H1-H3 into appropriate subsequences with the help of the objective values. Different from the iterative framework proposed in \cite{QianPan23} 
for the GLL-type nonmonotone decrease condition, the nonmonotone decrease condition in H1 does not provide any hint on the division. To overcome the difficulty, we make full use of the recursion formula on $C_k$ to skillfully divide $\{x^k\}_{k\in\mathbb{N}}$ into two subsequences, and then establish a recursion relation of the auxiliary sequence $\{\Xi_k\}_{k\in\mathbb{N}}$ in Lemma \ref{Lemma-Bound-Xi}. 
Due to the remarkable difference between these two nonmonotone decrease conditions
as showcased in the example \cite[(1.3)]{ZhangH04}, the convergence analysis here is also completely different from the one developed in \cite{QianPan23}.
A technical comparison on the sequence splitting scheme proposed in this paper
and that of \cite{QianPan23} can be found in Remark~\ref{Remark-GLL-ZH}. 

As discussed in \cite{Attouch10,Ioffe09}, there are a large number of nonconvex and nonsmooth optimization problems involving the KL functions. Furthermore, by \cite[Proposition 1]{Robinson81} and \cite[Proposition 2 (i) \& Remark 1 (b)]{LiuPan24}, the piecewise linear-quadratic KL functions with the composite structure of \cite{LiuPan24} necessarily satisfy the KL property of exponent $1/2$. Thus, the obtained convergence results will have a wide range of applications. As a demonstration, Section~\ref{sec4} achieves the full convergence of the iterate sequences produced by two existing algorithms: the proximal gradient method of 
 \cite{Marchi23} for composite optimization and
 the Riemannian gradient method of \cite{Wen13, Oviedo22} for
 Riemannian optimization. This result is new and enhances the existing convergence results
 for the two methods. Section~\ref{sec5} concludes the paper. 

\section{Notation and preliminary results}\label{sec2} 

The quantities and notation for describing H1 to H3 will be reserved for use throughout the paper
and they are: $\{a_k\}$, $\{b_k\}$, $\{C_k\}$, $\{\tau_k\}$, $\{\varepsilon_k\}$, $\overline{B}$,
$\underline{a}$, $\tau$ and $k_1$. We also use other (global) notations. Let $m$ be the smallest positive integer such that $\sqrt{\tau}(m\!-\!k_1\!-\!1)\ge(1\!+\!\sqrt{1\!-\!\tau})(2k_1\!+\!1)\sqrt{m}$.
Obviously, such $m$ exists and $m >k_1\!+\!1$. For each $k\in\mathbb{N}$, write $\ell(k):=k+m-1$ and $\Xi_{k-1}\!:=\!\sqrt{C_{k-1}\!-\!C_{k}}$ (we will prove in Lemma~\ref{lemma1-Phi} (i) that $\{C_k\}$ is a nonincreasing sequence and hence $\Xi_k$ is well defined). For a real number $t$, the floor operator $\lfloor t \rfloor$ is the largest integer not greater than $t$ and
$t_+ := \max\{0,t\}$ (the nonnegative part of $t$). 
For a proper $h\!:\mathbb{X}\!\to\overline{\mathbb{R}}$, denote by $\partial h(\overline{x})$ the (limiting) subdifferential of $h$ at $\overline{x}\in{\rm dom}\,h$, and for any $-\infty<\!\eta_1<\!\eta_2<\!\infty$, write $[\eta_1<h<\eta_2]\!:=\{x\in\mathbb{X}\,|\,\eta_1<h(x)<\eta_2\}$. The point $\overline{x}$ at which $0\in\partial h(\overline{x})$ is called a critical point of $h$, and the set of all critical points of $h$ is denoted by ${\rm crit}\,h$. 
 Now we introduce the formal definition of KL property (with exponent $\theta\in[0,1)$).
 \begin{definition}\label{KL-def}
 For a given $\eta\in(0,\infty ]$, denote by $\Upsilon_{\!\eta}$ the family of continuous concave $\varphi\!: [0, \eta )\rightarrow\mathbb{R}_{+}$ that is continuously differentiable on $(0,\eta)$ with $\varphi'(s)>0$ for all $s\in(0,\eta)$ and $\varphi(0)=0$. A proper function $h\!:\mathbb{X}\to\overline{\mathbb{R}}$ is said to have the KL property at $\overline{x}\in{\rm dom}\,\partial h$ if there exist $\eta\in(0,\infty]$,   a neighborhood $\mathcal{U}$ of $\overline{x}$, and a function $\varphi\in\Upsilon_{\!\eta}$ such that for all $x\in\mathcal{U}\cap\big[h(\overline{x})<h<h(\overline{x})+\eta\big]$,
 \[
  \varphi'(h(x)\!-\!h(\overline{x})){\rm dist}(0,\partial h(x))\ge 1.
 \]
 If $\varphi$ can be chosen as $\varphi(t)=ct^{1-\theta}$ with $\theta\in[0,1)$
  for some $c>0$, then $h$ is said to have the KL property of exponent $\theta$
  at $\overline{x}$. If $h$ has the KL property (of exponent $\theta$) at every point
  of ${\rm dom}\,\partial h$, then it is called a KL function (of exponent $\theta$).
 \end{definition}

Next we present two technical lemmas used for the subsequent analysis. The first one, proved in \cite{QianPan23}, says that if a nonnegative and nonincreasing sequence is bounded by a mixture of two special sequences related to the indices of the sequence, then the sequence itself can be bounded by one particular sequence related to its indices.
 
\begin{lemma}\label{lemma-sequence}
 (see \cite[Lemma 2.7]{QianPan23})\ Let $\{\beta_{l}\}_{l\in\mathbb{N}}\subset\mathbb{R}_{+}$ be a nonincreasing sequence with
  $\beta_{l}\le \mu_0\max\big\{l^{\frac{1-\varsigma}{1-2\varsigma}}, (\beta_{l-m_1}-\beta_{l})^{\frac{1-\varsigma}{\varsigma}}\big\}$ for all $l\ge\max\{\overline{l},m_1\}$, where $\overline{l}$ and $m_1$ are the nonnegative integers, and $\mu_0>0$ and $\varsigma\in(\frac{1}{2},1)$ are the constants. Then, there exists $\widehat{\mu}>0$ such that for all $l\ge\overline{l}$, $\beta_{l}\le\max\big\{\mu_0,\widehat{\mu}^{\frac{1-\varsigma}{1-2\varsigma}}\big\}\lfloor\frac{l-\overline{l}}{m_1+1}\rfloor^{\frac{1-\varsigma}{1-2\varsigma}}$.
 \end{lemma}

 The second one provides the desirable properties of $\{C_k\}_{k\in\mathbb{N}}$ and $\{\Phi(x^k)\}_{k\in\mathbb{N}}$. 
\begin{lemma}\label{lemma1-Phi}
 Let $\{x^k\}_{k\in\mathbb{N}}$ be a sequence obeying condition H1. Then, the following claims hold.  
 \begin{itemize}
 \item[(i)] For each $k\in\mathbb{N}$, $\Phi(x^k)\le C_{k}\le C_{k-1}$, so the sequence $\{C_k\}_{k\in\mathbb{N}}$ is convergent.  

 \item[(ii)] The sequence $\{\Phi(x^k)\}_{k\in\mathbb{N}}$ is convergent and has the same limit as $\{C_k\}_{k\in\mathbb{N}}$.

 \item[(iii)] $\lim_{k\to\infty}(x^{k+1}\!-\!x^{k})=0$. 
		
 \item[(iv)] If in addition $\{x^k\}_{k\in\mathbb{N}}$  complies with conditions H2-H3, then the sequences $\{C_k\}_{k\in\mathbb{N}}$ and $\{\Phi(x^k)\}_{k\in\mathbb{N}}$ both converge to $\Phi(\overline{x})$ with 
        $0\in\partial\Phi(\overline{x})$.
\end{itemize}
\end{lemma}
\begin{proof}
{\bf(i)-(iii)} For each $k\in\mathbb{N}$, from condition H1, it immediately follows that
 \begin{align}\label{Ck-monotone1}
 C_{k}=(1-\tau_{k})C_{k-1}+\tau_{k}\Phi(x^{k})
  &\le C_{k-1}-a_{k}\tau_{k}\|x^{k}-x^{k-1}\|^2\\
  &\le C_{k-1}-\underline{a}\tau\|x^{k}-x^{k-1}\|^2,
 \label{Ck-monotone2}
 \end{align}
 which implies $C_{k}\le C_{k-1}$ and $\Phi(x^k)\le C_{k-1}$. The latter, along with $\Phi(x^{k})=(1-\tau_k)(\Phi(x^k)-C_{k-1})+C_k$, leads to $\Phi(x^k)\le C_k$. Recall that $\Phi$ is bounded below on ${\rm dom}\,\Phi$ and $\{x^k\}_{k\in\mathbb{N}}\subset{\rm dom}\,\Phi$. The recursion relation of $C_k$ in condition H1 implies that $\{C_k\}_{k\in\mathbb{N}}$ is lower bounded, so its convergence follows the nonincreasing. Note that $\Phi(x^k)-C_{k-1}=\frac{1}{\tau_k}(C_k-C_{k-1})$ and $\tau_k\in[\tau,1]$ for each $k\in\mathbb{N}$. Then, 
 the sequence $\{\Phi(x^k)\}_{k\in\mathbb{N}}$ is convergent and has the same limit as $\{C_k\}_{k\in\mathbb{N}}$. From inequality \eqref{Ck-monotone2} and the convergence of $\{C_k\}_{k\in\mathbb{N}}$, it is immediate to obtain $\lim_{k\to\infty}(x^{k+1}\!-\!x^{k})=0$.

 \noindent
 {\bf(iv)} Combining H3 with the lower semicontinuity of $\Phi$ yields that $\lim_{j\to\infty}\Phi(x^{k_j})=\Phi(\overline{x})$. Along with part (ii), we have $\lim_{k\to\infty} C_k =\lim_{k\to\infty}\Phi(x^k)=\Phi(\overline{x})$, so that  
\begin{equation} \label{Eq-Ck}
	\lim_{k\to\infty}\sum_{i= k- k_1}^{k+k_1}\Xi_{i-1}=\lim_{k\to\infty} \sum_{i= k- k_1}^{k+k_1} \sqrt{C_{i-1} - C_i} 
	= 0 .
\end{equation}
 From condition H2, for each $k\ge k_1$, there exists $w^{k}\in\partial\Phi(x^{k})$ satisfying the relations
 \begin{align}\label{Eq-wk}
  \|w^{k}\| & \le\frac{1}{b_{k}}\sum_{i=k-k_1}^{k+k_1}\|x^{i}\!-\!x^{i-1}\|
  +\varepsilon_{k} \stackrel{\eqref{Ck-monotone1}}{\le}\frac{1}{\sqrt{\tau}b_k}\sum_{i=k-k_1}^{k+k_1}\frac{1}{\sqrt{a_{i}}}\Xi_{i-1}+\varepsilon_{k}\nonumber\\
  &\le \left( \frac{1}{\sqrt{\tau}b_k}\sum_{i=k-k_1}^{k+k_1}\frac{1}{\sqrt{a_{i}}}
  \right)
  \sum_{i=k-k_1}^{k+k_1}\Xi_{i-1} 
  +\varepsilon_{k} 
  \stackrel{ \eqref{abk-cond}}{\le} \frac{\overline{B} }{ \sqrt{\tau} }
  \sum_{i=k-k_1}^{k+k_1}\Xi_{i-1}
  +\varepsilon_{k} .
  \end{align}
 Passing the limit $k\to\infty$ to the above inequality and using \eqref{Eq-Ck} lead to $\lim\limits_{k\to\infty}w^{k}\!=\!0$. Recall that $\partial\Phi$ is outer semicontinuous at $\overline{x}$ with respect to $\Phi$-attentive convergence $x\xrightarrow[\Phi]{}\overline{x}$ (i.e., $x\to\overline{x}$ and $\Phi(x)\to\Phi(\overline{x})$) by \cite[Proposition 8.7]{RW98}. Together with $\lim_{k\to\infty}\Phi(x^k)=\Phi(\overline{x})$ and $\lim_{j\to\infty}x^{k_j}=\overline{x}$, we obtain $0\in\partial\Phi(\overline{x})$. 
 \end{proof}

The inequality \eqref{Ck-monotone2} has a useful implication when there are many terms. For any given indices $k$ and $k'$ with $k'\ge k$, it holds
\begin{equation} \label{xk-bound}
 \sum_{i=k}^{k'} \| x^i - x^{i-1} \|
 \le \frac{1}{\sqrt{\underline{a}\tau }} \sum_{i=k}^{k'} \sqrt{C_{i-1} - C_i }
 = \frac{1}{\sqrt{\underline{a}\tau }} \sum_{i=k}^{k'} \Xi_{i-1}.
\end{equation} 
This inequality will be often used in the convergence analysis of the next section.

\section{Main results}\label{sec3}

In this section, we report two results for the iterate sequence satisfying H1-H3. The first one is about its full convergence and the second is about its local convergence rate. We report them in two subsections. For convenience, in the rest of this paper, let $\{x^k\}_{k\in\mathbb{N}}$ be a sequence satisfying conditions H1-H3. 

 \subsection{Full convergence}\label{sec3.1}
 
Previously, we have proved that the two consecutive  terms of $\{x^k\}_{k\in\mathbb{N}}$ gets arbitrarily close as $k$ increases. 
If we further assume that $\Phi$ has the KL property, the sequence becomes a Cauchy sequence. 
We achieve this by establishing $\sum_{i=0}^\infty \Xi_i<\infty$. The result $\sum_{i=0}^\infty \| x^{i+1}-x^i \| < \infty$ then follows \eqref{xk-bound}. 

We first establish an iterative bound for the sums of $\Xi_i$. To this end, we define
\begin{equation*}
 \mathcal{K}_1:=\big\{k\in\mathbb{N}\ |\ \Phi(x^k)\le C_{k+m}\big\}\ \ {\rm and}\ \  \mathcal{K}_2:=\big\{k\in\mathbb{N}\ |\ \Phi(x^k)>C_{k+m}\big\}.
\end{equation*}
When $\Phi$ is assumed to have the KL property at $\overline{x}$, by Definition \ref{KL-def}, there exist $\delta>0,\eta>0$ and $\varphi\in\Upsilon_{\!\eta}$ such that for all $x\in\mathbb{B}(\overline{x},\delta)\cap[\Phi(\overline{x})<\Phi<\Phi(\overline{x})+\eta]$,
\begin{equation}\label{KL-ineq0}
 \varphi'(\Phi(x)-\Phi(\overline{x})){\rm dist}(0,\partial \Phi(x))\ge 1.
\end{equation}
 Recall that $C_k\ge\Phi(\overline{x})$ for each $k\in\mathbb{N}$ by Lemma \ref{lemma1-Phi} (i) and (iv). Hence, the following
 \[
 \Gamma_{k,k+m}:=\varphi\big(\Phi(x^{k})-\Phi(\overline{x})\big)-\varphi\big(C_{k+m}-\Phi(\overline{x})\big)\ \ {\rm for}\ k\in \mathcal{K}_2
 \]
 is well defined. Now we are ready to state the bound on some partial sums of $\Xi_i$.
\begin{lemma} \label{Lemma-Bound-Xi}
 Suppose that $\Phi$ has the KL property at $\overline{x}$. Let $\delta>0,\eta>0$ and $\varphi\in\!\Upsilon_{\!\eta}$ be chosen to satisfy \eqref{KL-ineq0}. Pick any $\rho \in (0, \delta)$. Then, for each $k\in\mathbb{N}$ with $x^k\in\mathbb{B}(\overline{x},\rho)$ and $\Phi(x^k)<\Phi(\overline{x})+\eta$, the following inequality holds with $\widehat{c}:=\frac{1}{2}\big(
\overline{B}/\sqrt{\tau} +1 \big)$:
\begin{align}\label{aim-ineq}
	\frac{\sqrt{\tau}}{\sqrt{m}}\sum_{i=k}^{\ell(k)}\Xi_{i}
	&\le\left\{\begin{array}{cc}
		\sqrt{1\!-\!\tau}\,\Xi_{k-1} & {\rm if}\ k\in \mathcal{K}_1, \\
		(\frac{1}{2}\!+\!\sqrt{1\!-\!\tau})\,\sum_{i=k-k_1}^{k+k_1}\Xi_{i-1}+\varepsilon_{k}+\widehat{c}\,\Gamma_{k,k+m} & {\rm if}\ k\in \mathcal{K}_2.
	\end{array}\right.
\end{align}
\end{lemma} 
\begin{proof}
 We will frequently use the following bound for every $k\in\mathbb{N}$
 \begin{equation} \label{squareroot-bound}
 \frac{1}{m}\sum_{i=k}^{\ell(k)}\Xi_i
 = \frac{1}{m}{\sum_{i=k}^{\ell(k)}}\sqrt{C_{i} - C_{i+1} } 
  \le\sqrt{{\sum_{i=k}^{\ell(k)}}\frac{1}{m}(C_{i} - C_{i+1})  }
 = \frac{1}{\sqrt{m}} \sqrt{C_{k} - C_{k+m}},
\end{equation}
where the inequality is due to the concavity of the function $\mathbb{R}_{+}\ni t\mapsto\sqrt{t}$. Fix any $k\!\in\!\mathbb{N}$ with $x^k\!\in\!\mathbb{B}(\overline{x},\rho)$ and $\Phi(x^k)\!<\!\Phi(\overline{x})\!+\!\eta$. We proceed the arguments by two cases.
	
 \medskip
 \noindent
 {\bf Case 1: $k\in\mathcal{K}_1$.} In this case, $\Phi(x^k)\le C_{k+m}$. By the equality in \eqref{Ck-monotone1} and $\tau_k\in[\tau,1]$, 
 \begin{align*}
 C_k-C_{k+m}&=(1-\tau_k)C_{k-1}+\tau_k\Phi(x^k)-C_{k+m}
 \le(1-\tau_k)C_{k-1}+(\tau_k-1)C_{k+m}\\
 &= (1-\tau_k)(C_{k-1}-C_{k+m})\le (1\!-\!\tau)(C_{k-1}-C_{k}+C_{k}-C_{k+m})\\
 &=(1\!-\!\tau)(\Xi_{k-1}^2+C_{k}-C_{k+m}), 
 \end{align*}
 which implies that $\sqrt{\tau(C_k-C_{k+m})}\le\!\sqrt{1\!-\!\tau}\,\Xi_{k-1}$. Combining this inequality with the above \eqref{squareroot-bound}, we obtain the inequality \eqref{aim-ineq} for $k\in\mathcal{K}_1$. 

 \noindent
 {\bf Case 2: $k\in\mathcal{K}_2$.} Now $\Phi(\overline{x})\le C_{k+m}<\Phi(x^k)<\Phi(\overline{x})+\eta$. Using \eqref{KL-ineq0} with $x=x^k$ yields that $\varphi'(\Phi(x^k)-\Phi(\overline{x})){\rm dist}(0,\partial\Phi(x^{k}))\ge 1$, which by condition H2 implies that
 \begin{equation} \label{KL-Bound}
 \left(\frac{1}{b_{k}}\sum_{i=k-k_1}^{k+k_1}\|x^{i}\!-\!x^{i-1}\|
  +\varepsilon_{k} \right)
 \varphi'\big(\Phi(x^{k})-\Phi(\overline{x})\big)\ge 1. 
 \end{equation}
 From the definition of $\Gamma_{k,k+m}$ and the concavity of $\varphi$ on $[0,\eta)$, we have $\Gamma_{k,k+m}\ge \varphi'\big(\Phi(x^{k})\!-\!\Phi(\overline{x})\big)\big(\Phi(x^{k})\!-\!C_{k+m} \big)$. This along with \eqref{KL-Bound} leads to the inequalities
 \begin{align*}
 \Phi(x^{k})\!-\!C_{k+m}
  \le\!\Big(\frac{1}{b_{k}}\sum_{i=k-k_1}^{k+k_1}\!\!\|x^{i}\!-\!x^{i-1}\|+\varepsilon_{k}\Big)\Gamma_{k,k+m}\!\stackrel{ \eqref{Eq-wk}}{\le}\!
		\Big(\frac{\overline{B} }{ \sqrt{\tau} }\!\!
  \sum_{i=k-k_1}^{k+k_1}\!\!\Xi_{i-1}
  +\varepsilon_{k}\Big)\Gamma_{k,k+m}.
 \end{align*} 
 In addition, by the equality in \eqref{Ck-monotone1}, $\Phi(x^k)-C_{k+m}=C_k\!-\!C_{k+m}\!-\!\frac{1-\tau_k}{\tau_k}(C_{k-1}\!-\!C_{k})$. Along with the above inequality, the definition of $\overline{B}$ in \eqref{abk-cond}, and $0<\tau_k\le 1$, we have
\begin{align*}
 \tau_k(C_k\!-\!C_{k+m}) &\le (1-\!\tau_k)(C_{k-1} - C_{k})\!+
 \left(\frac{\overline{B} }{ \sqrt{\tau} }
 \sum_{i=k-k_1}^{k+k_1}\!\!\Xi_{i-1}
  +\varepsilon_{k}\right)\Gamma_{k,k+m}.
\end{align*}
 By the definition of $\widehat{c}$, we have $\overline{B}/\sqrt{\tau} = 2 \widehat{c}-1$. With this in mind, recalling that $\Xi^2_{k-1}=C_{k-1} - C_{k}$, we continue to bound the term $\sqrt{\tau_k (C_k - C_{k+m})}$:
 \begin{align*}
 \sqrt{\tau_k(C_k\!-\!C_{k+m})}&\le
 \sqrt{(1\!-\!\tau_k)\Xi_{k-1}^2+\big((2\widehat{c}\!-\!1)\textstyle{\sum_{i=k-k_1}^{k+k_1}}\!\Xi_{i-1}+\varepsilon_{k}\big)\Gamma_{k,k+m}}\\
 &\le \!\sqrt{1\!-\!\tau_k}\Xi_{k-1}\!
		+\!\sqrt{(2\widehat{c}\!-\!1)\textstyle{\sum\limits_{i=k-k_1}^{k+k_1}}\Xi_{i-1}\Gamma_{k,k+m}}\!+\!\sqrt{\varepsilon_{k}\Gamma_{k,k+m}}\\
 &=\sqrt{1\!-\!\tau_k}\Xi_{k-1}\! +
		2 \sqrt{ \Big( (\widehat{c}\!-\!1/2) \Gamma_{k,k+m} \Big) \Big(\frac{1}{2}\textstyle{\sum\limits_{i=k-k_1}^{k+k_1}}\Xi_{i-1}\Big)}\\
 &\quad + 2\sqrt{(\varepsilon_{k}/2) (\Gamma_{k,k+m}/2)} \\
 &\le\!\sqrt{1\!-\!\tau_k}\Xi_{k-1}\!+\!\frac{1}{2}\sum_{i=k-k_1}^{k+k_1}\Xi_{i-1}\!+\!\frac{1}{2}\varepsilon_{k}\!+\!\widehat{c}\,\Gamma_{k,k+m}\\
 &\le\!\Big(\frac{1}{2}\!+\!\sqrt{1\!-\!\tau}\Big)\sum_{i=k-k_1}^{k+k_1}\!\Xi_{i-1}+\varepsilon_{k}+\widehat{c}\,\Gamma_{k,k+m},
 \end{align*}
 where the first inequality is by the definition of $\widehat{c}$, and the second is due to $\sqrt{\alpha\!+\!\beta}\!\le\!\sqrt{\alpha}\!+\!\sqrt{\beta}$ for $\alpha,\beta\!\ge\! 0$. The penultimate inequality is due to the twice use of the fact	$2\sqrt{\alpha \beta}\le\alpha \!+\! \beta$ for $\alpha,\! \beta \!\ge \!\!0$. Using \eqref{squareroot-bound} again yields the inequality \eqref{aim-ineq} for $k\in\mathcal{K}_2$.
\end{proof}

{\begin{remark} \label{Remark-GLL-ZH}
To fully appreciate the innovative role that the sequence splitting in $\mathcal{K}_1$ and $\mathcal{K}_2$ plays in Lemma~\ref{Lemma-Bound-Xi} and also in future results,
we like to make a technical note to highlight its difference from the splitting scheme
used in \cite{QianPan23} for the GLL-type nonmonotone line search.
We recall the search criterion at $x^k$ of GLL-type is
\[
  \Phi(x^{k+1}) + a \| x^{k+1} - x^k \|^2 \le \max_{j =[k-m_0]_+, \ldots, k  } \Phi(x^j) ,
\]
where $m_0 \ge 0$ is a given integer and $a >0$ is a constant. 
Let $\widetilde{\ell}(k)$ be the maximum index in $\arg\max _{j =[k-m_0]_+, \ldots, k  } \Phi(x^j)$.
Apparently, we have
\[
  \Phi(x^{k+1}) \le \Phi( x^{\widetilde{\ell}(k)}) - a \| x^{k+1} - x^k \|^2 .
\]
The quality of $\Phi(x^{k+1})$ is further measured against the value of $\Phi( x^{\widetilde{\ell}(k+1)})$. Hence, we define
\[
 \widetilde{\mathcal{K}}_1 := \left\{ 
   k \in \mathbb{N} \ | \
   \Phi(x^{k+1}) \le \Phi( x^{\widetilde{\ell}(k+1)}) - \frac a2 \| x^{k+1} - x^k \|^2
 \right\} .
\]
This sequence was naturally suggested by the GLL-type search rule and was used in \cite{QianPan23}
to derive convergence properties on the sequence $\widetilde{\mathcal{K}}_1$. 
In contrast, the search in H1 uses a weighted average of functional values contained in $C_k$. 
There is no obvious way to refer to particular functional 
values $\{\Phi(x^k)\}$  as they are averaged. 
Instead, we collect in $\mathcal{K}_1$ 
all indices $k$ with $\Phi(x^k) \le C_{k+m}$. 
Note that $m \ge 2$ in our setting. This means that the value $C_{k+m}$ contains iterate information
upto $x^{k+m}$ (the future iterates $x^{k+2}, \ldots, x^{k+m}$ were involved). 
In contrast, $\widetilde{\mathcal{K}}_1$ only used the information upto $x^{k+1}$, which was already 
calculated at $x^k$. 
The future vs the present information being involved in the sequence splitting requires a completely
different set of analysis to derive the respective convergence result. 
Further difference between our results and those from \cite{QianPan23} is discussed in Remark~\ref{remark-result}.	
\end{remark}}

\begin{theorem}\label{KL-converge}
 If $\Phi$ has the KL property at $\overline{x}$, then $\sum_{k=0}^\infty\|x^{k+1}\!-\!x^k\|<\infty$, and consequently the sequence $\{x^k\}_{k\in\mathbb{N}}$ converges to $\overline{x}$ that is a critical point of $\Phi$.
\end{theorem}
\begin{proof}
 From $\lim_{k\to\infty}\Phi(x^k)\!=\!\Phi(\overline{x})$, there exists $\mathbb{N}\ni\widehat{k}\!>\!k_1$ such that for all $k\ge\widehat{k}$, $\Phi(x^k)<\Phi(\overline{x})\!+\!\eta$. Recall that $\lim_{k\to\infty}C_k\!=\!\Phi(\overline{x})$ and $\lim_{k\to\infty}(x^{k+1}\!-\!x^k)\!=\!0$. Along with condition H3, $\sum_{k=1}^{\infty}\varepsilon_k<\infty$, and the continuity of $\varphi$ on $[0,\eta)$, if necessary by increasing $\widehat{k}$, we have
 \begin{equation}\label{temp-ineq33}
  \|x^{\widehat{k}}-\overline{x}\|+\frac{4(2k_1\!+\!1)}{\sqrt{\underline{a}\tau}}\!\!\sum_{j=\widehat{k}-k_1}^{\ell(\widehat{k})}\!\!\Xi_{j-1}+\frac{2\widehat{c}}{\sqrt{\underline{a}\tau}}\!\sum_{j=\widehat{k}}^{\ell(\widehat{k})}\varphi\big(C_j-\Phi(\overline{x})\big)+\frac{2}{\sqrt{\underline{a}\tau}}\sum_{j=\widehat{k}}^{\infty}\varepsilon_j<\rho,
 \end{equation}
 where $\rho$ is the same as in Lemma \ref{Lemma-Bound-Xi}. Thus, by Lemma \ref{Lemma-Bound-Xi}, inequality \eqref{aim-ineq} holds for  $k=\widehat{k}$. We claim that the desired conclusion holds if for each $\nu\ge\ell(\widehat{k})$, 
\begin{equation} \label{aim-ineq31}
\left\{
 \begin{array}{ll}
  x^{\nu}\in\mathbb{B}(\overline{x},\rho),  & \\ [1ex]
  \sqrt{\tau m} \sum_{j=\ell(\widehat{k})}^{\nu}\!\Xi_{j}\!\!\!
   & \le \big(\frac{1}{2}\!\!+\!\!\sqrt{1-\tau}\big)(2k_1\!+\!1)\!\!\sum_{j=\widehat{k}\!-\!k_1}^{\nu+1}\!\!\Xi_{j-1}+\!\sum_{j=\widehat{k}}^{\nu}\varepsilon_j \\ [1ex]
   &\quad +\widehat{c}\sum_{j=\widehat{k}}^{\ell(\widehat{k})}\varphi\big(C_j\!-\!\Phi(\overline{x})\big).
 \end{array}\right.	
\end{equation} 
 Indeed, by the definition of $m$, $\sqrt{m\tau}\ge\frac{\sqrt{\tau}(m-k_1\!-\!1)}{\sqrt{m}}\ge(1\!+\!\sqrt{1\!-\!\tau})(2k_1\!+\!1)$, so that
 \[
 \widehat{c}_1:=\sqrt{m\tau}\!-\!(\frac{1}{2}\!+\!\sqrt{1\!-\!\tau})(2k_1\!+\!1)
 \ge \frac{2k_1+1 }{2} \ge \frac 12.
 \]
 According to this inequality, the inequality in \eqref{aim-ineq31} implies that
 \begin{align}\label{temp-ineq32}
 \widehat{c}_1\!\!\sum_{j=\ell(\widehat{k})}^{\nu}\!\!\Xi_{j}
  \!\le\!\frac{(1\!+\!2\sqrt{1\!-\!\tau})(2k_1\!+\!1)}{2}\!\!\!\sum_{j=\widehat{k}-k_1}^{\ell(\widehat{k})}\!\!\!\Xi_{j\!-\!1}\!+\widehat{c}\sum_{j=\widehat{k}}^{\ell(\widehat{k})}\varphi\big(C_j\!-\!\Phi(\overline{x})\big)\!+\!\sum_{j=\widehat{k}}^{\nu}\varepsilon_j.
 \end{align} 
 Passing the limit $\nu\to\infty$ to the both sides of \eqref{temp-ineq32} leads to 
 $\sum_{k=1}^\infty\Xi_{k}<\infty$, which by \eqref{xk-bound} implies that $\sum_{k=0}^\infty\|x^{k+1}\!-\!x^k\|<\infty$. By condition H3, $\{x^k\}_{k\in\mathbb{N}}$ converges to $\overline{x}$.

 Next we prove by induction that the claim \eqref{aim-ineq31} holds for every $\nu\ge\ell(\widehat{k})$. Indeed, the above \eqref{temp-ineq33} implies that $\|x^{\widehat{k}}-\overline{x}\|<\rho$. For any $\widehat{k}+1\le \nu \le \ell(\widehat{k})$, using $\|x^{\nu}-\overline{x}\|\le\|x^{\widehat{k}}-\overline{x}\|+\|x^{\nu}-x^{\widehat{k}}\|$ and the previous \eqref{xk-bound} and \eqref{temp-ineq33} leads to 
 \[
 \|x^{\nu}-\overline{x}\|\le\|x^{\widehat{k}}-\overline{x}\|
 +\sum_{j=\widehat{k}+1}^{\nu}\|x^{j}-x^{j-1}\|
 \stackrel{\eqref{xk-bound}}{\le} \|x^{\widehat{k}}-\overline{x}\|+\frac{1}{\sqrt{\underline{a}\tau}}\!\!\sum_{j=\widehat{k}+1}^{\nu}\!\Xi_{j-1}
 \stackrel{\eqref{temp-ineq33}}{<} \rho.
 \]
Thus, we have $x^\nu \in \mathbb{B}(\overline{x},\rho)$ for 
$\widehat{k}\le \nu \le \ell(\widehat{k})$. Together with $\Phi(x^k)<\Phi(\overline{x})\!+\!\eta$ for all $k\ge\widehat{k}$, invoking  Lemma~\ref{Lemma-Bound-Xi} shows that inequality \eqref{aim-ineq} holds for $x^{\widehat{k}}$, $\cdots$, $x^{\ell(\widehat{k})}$. Summing inequality \eqref{aim-ineq} from $\widehat{k}$ to $\ell(\widehat{k})$ and using the nonnegativity of $\varphi$ leads to
 \begin{align}\label{initial-bound}
  \frac{\sqrt{\tau}}{\sqrt{m}}
  \underbrace{\sum_{j=\widehat{k}}^{\ell(\widehat{k})}\sum_{i=j}^{\ell(j)}\Xi_{i}}_{:=\Delta_1}
  &\le \left(\frac{1}{2}\!+\!\sqrt{1\!-\!\tau}\right)
  \sum_{j=\widehat{k}}^{\ell(\widehat{k})}\sum_{i=j-k_1}^{j+k_1}\!\!\Xi_{i-1}
  +\sum_{j=\widehat{k}}^{\ell(\widehat{k})}\varepsilon_{j}+\widehat{c}\!\!\sum_{\mathcal{K}_2\ni j=\widehat{k}}^{\ell(\widehat{k})}\!\varphi\big(\Phi(x^j)\!-\!\Phi(\overline{x})\big)\nonumber\\
  &\le\left(\frac{1}{2}\!+\!\sqrt{1\!-\!\tau}\right)
  \underbrace{\sum_{j=\widehat{k}}^{\ell(\widehat{k})}\sum_{i=j-k_1}^{j+k_1}\!\!\Xi_{i-1}}_{:= \Delta_2}
  +\sum_{j=\widehat{k}}^{\ell(\widehat{k})}\varepsilon_{j}+\widehat{c}\sum_{j=\widehat{k}}^{\ell(\widehat{k})}\varphi\big(C_j\!-\!\Phi(\overline{x})\big),
 \end{align}
 where the second inequality is due to $\Phi(x^j)\le C_j$ by Lemma \ref{lemma1-Phi} (i) and the increasing of $\varphi$. 
Both $\Delta_1$ and $\Delta_2$ contain many terms. We simply them. 
Recall that $m > k_1+1$ by the definition of $m$ and 
$\ell(\widehat{k}) = \widehat{k} + m -1$.
For $\Delta_1$, there are at least as many as $m$ terms of $\Xi_{\ell(\widehat{k})  }$. For each $j\in\{\widehat{k},\widehat{k}+1,\ldots,\ell(\widehat{k})\}$, we  consider the terms in the sum $\sum_{i=j}^{\ell(j)}\Xi_i$ that have indices higher than $\ell(\widehat{k})$. Note that $\ell(j)\!\ge\!\ell(\widehat{k})\!+\!k_1\!+\!1$ for $j\!=\!\widehat{k}\!+\!k_1\!+\!1, \ldots, 
\widehat{k} \!+\! m \!-\!1$. We have 
\(
  \sum_{i=j}^{\ell(j)} \Xi_i\ge\sum_{i= \widehat{k}+m}^{\ell(\widehat{k}) +k_1+1 } \Xi_i,
\)
and consequently,
\[
  \Delta_1 \ge m \Xi_{\ell(\widehat{k})  } + (m-k_1\!-\!1) \sum_{i= \widehat{k}+m}^{\ell(\widehat{k}) +k_1+1 } \Xi_i .
\]
For $\Delta_2$, we consider the lowest and highest index for $i$ and they are $i=\widehat{k}-k_1$ and $i= \ell(\widehat{k}) + k_1$. Therefore,
it holds
\[
  \Delta_2 \le (2k_1+1) \sum_{i= \widehat{k}-k_1 }^{\ell(\widehat{k}) + k_1 }
  \Xi_{i-1} .
\]
Substituting those bounds back to \eqref{initial-bound} and obtaining  the following inequalities:
 \begin{align}
  &\sqrt{\tau m}\,\Xi_{\ell(\widehat{k})}+\frac{\sqrt{\tau}(m\!-\!k_1\!-\!1)}{\sqrt{m}}\sum_{j=\widehat{k}+m}^{\ell(\widehat{k})+k_1+1}\!\!\Xi_{j} \nonumber\\
  &\le
   \left(\frac{1}{2}\!+\!\sqrt{1\!-\!\tau}\right)(2k_1\!+\!1)\sum_{j=\widehat{k}-k_1}^{\ell(\widehat{k})+k_1}\!\!\Xi_{j-1}+\sum_{j=\widehat{k}}^{\ell(\widehat{k})}\varepsilon_{j}+\widehat{c}\sum_{j=\widehat{k}}^{\ell(\widehat{k})}\varphi\big(C_j\!-\!\Phi(\overline{x})\big) \nonumber \\
   & \le\left(\frac{1}{2}\!+\!\sqrt{1\!-\!\tau}\right)(2k_1\!+\!1)
   \left[~\sum_{j=\widehat{k}-k_1}^{\ell(\widehat{k})+1}\!\!\Xi_{j-1}
   + \sum_{j=\ell(\widehat{k})+2}^{\ell(\widehat{k}) +k_1+2}\!\!\Xi_{j-1}
   \right] \nonumber \\
   &~~~ +\sum_{j=\widehat{k}}^{\ell(\widehat{k})}\varepsilon_{j}+\widehat{c}\sum_{j=\widehat{k}}^{\ell(\widehat{k})}\varphi\big(C_j\!-\!\Phi(\overline{x})\big) . \label{Wanted-1}
 \end{align}
 Recall that $\ell(\widehat{k})+1 = \widehat{k}+m$ and $\frac{\sqrt{\tau}(m\!-\!k_1\!-\!1)}{\sqrt{m}}\ge(1\!+\!\sqrt{1\!-\!\tau})(2k_1\!+\!1)$. The above inequality implies that the inequality in \eqref{aim-ineq31} holds for $\nu=\ell(\widehat{k})$, so the claimed 
 \eqref{aim-ineq31} holds for $\nu=\ell(\widehat{k})$. Now assume that the claimed \eqref{aim-ineq31} holds for some $\nu\ge\ell(\widehat{k})$. We argue that it holds for $\nu+1$. From the triangle inequality and \eqref{xk-bound}, it holds
 \[
 \|x^{\nu+1}\!-\overline{x}\|\le \|x^{\widehat{k}}\!-\overline{x}\|+\sum_{j=\widehat{k}}^{\nu}\|x^{j+1}\!-x^j\|
 \le\|x^{\widehat{k}}-\overline{x}\|+\frac{1}{\sqrt{\underline{a}\tau}}\Big[\sum_{j=\widehat{k}}^{\ell(\widehat{k})-1}\Xi_{j}+\!\!\sum_{j=\ell(\widehat{k})}^{\nu}\Xi_{j}\Big]. 
 \]
 Note that \eqref{temp-ineq32} holds for this $\nu$. Together with the above inequality, it follows that  
 \begin{align*}
  \|x^{\nu+1}-\overline{x}\|
 &\le \|x^{\widehat{k}}-\overline{x}\|
  + \frac{1}{\sqrt{\underline{a}\tau}}
  \sum_{j=\widehat{k}}^{\ell(\widehat{k})-1}\!\!\Xi_j\!
  +
  \frac{(1+2\sqrt{1-\tau})(2k_1\!+\!1)}{\sqrt{\underline{a}\tau}}\!\!\sum_{j=\widehat{k}-k_1}^{\ell(\widehat{k})}\!\!\Xi_{j-1}\\
  &\qquad\!+\!\frac{2}{\sqrt{\underline{a}\tau}}\sum_{j=\widehat{k}}^\nu\varepsilon_j+\!\frac{2\widehat{c}}{\sqrt{\underline{a}\tau}}\sum_{j=\widehat{k}}^{\ell(\widehat{k})}
  \varphi\big(C_j\!-\!\Phi(\overline{x})\big)
  \quad\ \mbox{(also used $\widehat{c}_1 \ge 1/2$.)}
 \end{align*}
 Recall that $\tau\in(0,1]$, so $1\!+\!2\sqrt{1\!-\!\tau}<3$. Then, we have
 \begin{align*}
  \|x^{\nu+1}-\overline{x}\|
  &\le\|x^{\widehat{k}}-\overline{x}\|
  \!+\!\frac{3(2k_1\!+\!1)}{\sqrt{\underline{a}\tau}}\!\!\sum_{j=\widehat{k}-k_1}^{\ell(\widehat{k})}\!\!\Xi_{j-1}\!+\!\frac{2}{\sqrt{\underline{a}\tau}}\sum_{j=\widehat{k}}^\nu\varepsilon_j 
   \\
  &\qquad +\!\frac{2\widehat{c}}{\sqrt{\underline{a}\tau}}\sum_{j=\widehat{k}}^{\ell(\widehat{k})}
  \varphi\big(C_j\!-\!\Phi(\overline{x})\big) \stackrel{\eqref{temp-ineq33}}{<} \rho.
 \end{align*}
 This shows $x^{\nu+1}\in\mathbb{B}(\overline{x},\rho)$, so the rest only proves the inequality in \eqref{aim-ineq31} for $\nu+1$. Note that Lemma~\ref{Lemma-Bound-Xi} holds for $x^{\widehat{k}}$ up to $x^{\nu+1}$. Summing \eqref{aim-ineq} from $\widehat{k}$ to $\nu+1$ yields
 \begin{align*}
  \frac{\sqrt{\tau}}{\sqrt{m}}
 \sum_{j=\widehat{k}}^{\nu+1}\sum_{i=j}^{\ell(j)}\Xi_{i}
  &\le \Big(\frac{1}{2}\!+\!\sqrt{1\!-\!\tau}\Big)\sum_{j=\widehat{k}}^{\nu+1}\sum_{i=j-k_1}^{j+k_1}\!\Xi_{i-1}+\sum_{j=\widehat{k}}^{\nu+1}\varepsilon_{j}\nonumber\\
  &\quad +\widehat{c}\!\sum_{\mathcal{K}_2\ni j=\widehat{k}}^{\nu+1}\!\big[\varphi\big(\Phi(x^j)\!-\!\Phi(\overline{x})\big)-\varphi\big(C_{j+m}\!-\!\Phi(\overline{x})\big)\big].
 \end{align*} 
 As $\Phi(x^j)\le C_j$ for each $j\in\mathbb{N}$ by Lemma \ref{lemma1-Phi} (i), the increasing and nonnegativity of $\varphi$ implies that
 \begin{equation*}
 \sum_{\mathcal{K}_2\ni j=\widehat{k}}^{\nu+1}\!\big[\varphi\big(\Phi(x^j)\!-\!\Phi(\overline{x})\big)-\varphi\big(C_{j+m}\!-\!\Phi(\overline{x})\big)\big]
 \le\sum_{j=\widehat{k}}^{\ell(\widehat{k})}\varphi(C_j-\Phi(\overline{x})).
 \end{equation*}
 From the above two inequalities, it immediately follows that
 \begin{equation}\label{Delta3-4}
  \frac{\sqrt{\tau}}{\sqrt{m}}
  \underbrace{\sum_{j=\widehat{k}}^{\nu+1}\sum_{i=j}^{\ell(j)}\Xi_{i}}_{:=\Delta_3}
  \le \Big(\frac{1}{2}\!+\!\sqrt{1\!-\!\tau}\Big)\underbrace{\sum_{j=\widehat{k}}^{\nu+1}\sum_{i=j-k_1}^{j+k_1}\!\Xi_{i-1}}_{:=\Delta_4}+\sum_{j=\widehat{k}}^{\nu+1}\varepsilon_{j}+\widehat{c}\sum_{j=\widehat{k}}^{\ell(\widehat{k})}\varphi(C_j-\Phi(\overline{x})).
 \end{equation} 
For $\Delta_3$, there are at least as many as $m$ terms of $\sum_{i=\ell(\widehat{k})}^{\nu\!+\!1}\Xi_{i}$. For each $j\in\!\{\widehat{k},\ldots,\nu+1\}$, we  consider the terms in the sum $\sum_{i=j}^{\ell(j)}\Xi_i$ that have indices higher than $\nu\!+1$. Note that $\ell(j)\ge\nu\!+\!k_1\!+\!2$ for $j=\nu\!+\!k_1\!-\!m\!+\!3, \ldots, 
\nu\!+\!1$. We have 
\(
  \sum_{i=j}^{\ell(j)} \Xi_i\ge\sum_{i= \nu+2}^{\nu+k_1+2} \Xi_i,
\)
and consequently,
\[
  \Delta_3 \ge m \sum_{i=\ell(\widehat{k})}^{\nu+1}\Xi_{i} + (m-k_1-1) \sum_{i= \nu+2}^{\nu+k_1+2} \Xi_i .
\]
For $\Delta_4$, we consider the lowest and highest index for $i$ and they are $i=\widehat{k}-k_1$ and $i= \nu+k_1+1$. Therefore,
it holds
\[
  \Delta_4 \le (2k_1+1) \sum_{i= \widehat{k}-k_1}^{\nu+k_1+1 }
  \Xi_{i-1} .
\]
Substituting those bounds back to \eqref{Delta3-4} and get the following
 \begin{align*}
  &\sqrt{\tau m}\!\!\sum_{j=\ell(\widehat{k})}^{\nu+1}\!\!\Xi_{j}\!+\!\frac{\sqrt{\tau}(m\!-\!k_1\!-\!1)}{\sqrt{m}}\!\sum_{j=\nu+2}^{\nu+k_1+2}\!\!\!\Xi_{j}\\
  &\le
   \Big(\frac{1}{2}\!+\!\sqrt{1\!-\!\tau}\Big)(2k_1\!+\!1)\!\!\sum_{j=\widehat{k}-k_1}^{\nu+k_1+1}\!\!\Xi_{j-1}+\sum_{j=\widehat{k}}^{\nu+1}\varepsilon_{j}
   +\widehat{c}\sum_{j=\widehat{k}}^{\ell(\widehat{k})}\varphi(C_j-\Phi(\overline{x})).
 \end{align*}
 This, along with $\frac{\sqrt{\tau}(m\!-\!k_1\!-\!1)}{\sqrt{m}}\ge(1\!+\!\sqrt{1\!-\!\tau})(2k_1\!+\!1)$, implies that the inequality in \eqref{aim-ineq31} holds for $\nu+1$. Thus, we complete the induction proof. 
\end{proof}

\begin{remark}
 By the convergence analysis of this section, when H2 is relaxed to 
 \begin{itemize}
 \item[{\bf H2'.}] there exist $k_1\in\mathbb{N}$ and  $\overline{\delta}>0$ such that for all $k\ge k_1$ and $x^k\in\mathbb{B}(\overline{x},\overline{\delta})$, ${\rm dist}(0,\partial\Phi(x^{k}))\!\le\!\frac{1}{b_{k}}\sum_{i=k-k_1}^{k+k_1}\|x^{i}\!-\!x^{i-1}\|+\varepsilon_{k}$ with $b_{k}\!>\!0$ and $\varepsilon_{k}\!\ge\!0$, 
 \end{itemize}
 any sequence $\{x^k\}_{k\in\mathbb{N}}$ satisfying H1, H2' and H3 still has the full convergence. 
 \end{remark}
\subsection{Convergence rate}\label{sec3.2}

 Next we establish the convergence rate of $\{x^k\}_{k\in\mathbb{N}}$ under the KL property of $\Phi$ with exponent $\theta\in(0,1)$ by Theorem \ref{KL-converge} and Lemma \ref{lemma-sequence}. We will require the relative error $\varepsilon_k$ in condition H2 to be eventually controlled by a sequence related to the exponent. 
 This requirement is certainly satisfied if we simply set $\varepsilon_k =0$. For this purpose, we first establish the convergence rate of $\{C_k\}_{k\in\mathbb{N}}$.
\begin{lemma}\label{KL-fval-rate}
 Suppose that $\Phi$ has the KL property of exponent $\theta\in(0,1)$ at $\overline{x}$. If there exist $\widetilde{k}_0\in\mathbb{N},\widetilde{\gamma}>0$ and $\widetilde{\varrho}\in(0,1)$ such that for all $k\ge\widetilde{k}_0$,
 \begin{equation}\label{rate-cond0}
 \varepsilon_k\le\left\{\begin{array}{cl}
  \widetilde{\gamma}\widetilde{\varrho}^k &{\rm if}\ \theta\in(0,1/2];\\
  \!\widetilde{\gamma}k^{\frac{\theta}{1-2\theta}}&{\rm if}\ \theta\in(1/2,1),
 \end{array}\right.
 \end{equation}
 then there exist $\gamma>0$ and $\varrho\in(0,1)$ such that for sufficiently large $k$,
 \begin{equation}\label{aim-ineq-rate0}
  \widetilde{\Lambda}_k:=C_k-\Phi(\overline{x})
  \le\left\{\begin{array}{cl}
  \gamma\varrho^{k} &{\rm if}\ \theta\in(0,1/2],\\
  \gamma k^{\frac{1}{1-2\theta}}&{\rm if}\ \theta\in(1/2,1).
 \end{array}\right.
\end{equation}
\end{lemma}
\begin{proof}
 Since $\Phi$ has the KL property of exponent $\theta\in(0,1)$ at $\overline{x}$, 
 for any $x\in\mathbb{B}(\overline{x},\delta)\cap[\Phi(\overline{x})<\Phi<\Phi(\overline{x})+\eta]$, the inequality \eqref{KL-ineq0} holds with $\varphi(t)=ct^{1-\theta}$ for some $c>0$. Along with $\lim_{k\to\infty}x^k=\overline{x}$ by Theorem \ref{KL-converge}, $\lim_{k\to\infty}\Phi(x^k)=\Phi(\overline{x})$ and condition H2, there is $\widetilde{k}\ge\max\{\ell(\widehat{k}),\widetilde{k}_0\}$ such that for all $k\ge\widetilde{k}$ with $\Phi(x^k)>\Phi(\overline{x})$,  
 \begin{align}\label{ineq-Phi0}
  (\Phi(x^{k})\!-\!\Phi(\overline{x}))^{\theta}\!&\le\!
    c(1\!-\!\theta)\Big(\frac{1}{b_k}\!\!\sum_{i=k-k_1}^{k+k_1}\!\!\!\|x^{i}\!-\!x^{i-1}\|\!+\!\varepsilon_k\Big)\!
  \!\stackrel{ \eqref{Eq-wk}}{\le}\! c(1\!-\!\theta)\Big(\frac{\overline{B} }{ \sqrt{\tau} }
  \!\!\sum_{i=k-k_1}^{k+k_1}\!\Xi_{i-1}
  \!+\!\varepsilon_{k}\Big)\nonumber\\
  &\le c(1\!-\!\theta)\max\big\{1,(2\widehat{c}-1)\big\}\Big[\sum_{i=k-k_1}^{k+k_1}\Xi_{i-1}\!+\!\varepsilon_k\Big],
 \end{align}
 where the third inequality is by the definition of $\widehat{c}$ appearing in Lemma \ref{Lemma-Bound-Xi}. By letting  $\widehat{c}_2(\theta):=(c(1\!-\!\theta)\max\big\{1,(2\widehat{c}-1)\big\})^{\frac{1}{\theta}}$, for all $k\ge\widetilde{k}$ with $\Phi(x^k)>\Phi(\overline{x})$, we have
 \[
   \Phi(x^{k})\!-\!\Phi(\overline{x})\le \widehat{c}_2(\theta)\Big[\sum_{i=k-k_1}^{k+k_1}\Xi_{i-1}\!+\!\varepsilon_k\Big]^{\frac{1}{\theta}},
 \]
 which, together with the equality in \eqref{Ck-monotone1} and $\tau_k\ge\tau$, implies that 
 \begin{align}\label{temp-keyineq0}
  \widetilde{\Lambda}_k=C_{k}-\Phi(\overline{x})&=(1-\tau_k)(C_{k-1}-\Phi(\overline{x}))+\tau_k(\Phi(x^{k})-\Phi(\overline{x}))\nonumber\\
  &\le(1-\tau)\widetilde{\Lambda}_{k-1}+\tau\widehat{c}_2(\theta)\Big[\sum_{i=k-k_1}^{k+k_1}\Xi_{i-1}\!+\!\varepsilon_k\Big]^{\frac{1}{\theta}}.
 \end{align}
 While for all $k\ge\widetilde{k}$ with $\Phi(x^k)\le\Phi(\overline{x})$, from the equality in \eqref{Ck-monotone1} and $\tau_k\ge\tau$, we have
\[
  \widetilde{\Lambda}_k=(1-\tau_k)(C_{k-1}-\Phi(\overline{x}))+\tau_k(\Phi(x^{k})-\Phi(\overline{x}))\le(1-\tau)\widetilde{\Lambda}_{k-1}.
 \]
 Thus, the inequality \eqref{temp-keyineq0} holds for all $k\ge\widetilde{k}$. We proceed the proof by two cases.
 
 \noindent
  {\bf Case 1: $\theta\in(0,\frac{1}{2}]$.} Note that 
  $\lim_{k\to\infty}\sum_{i=k-k_1}^{k+k_1}\Xi_i=0$ and  $\lim_{k\to\infty}\varepsilon_k=0$. Hence, $\sum_{i=k-k_1}^{k+k_1}\Xi_i\!+\varepsilon_k\!<\!0.5$ for all $k\ge\widetilde{k}$ (if necessary by increasing $\widetilde{k}$). For any $k\ge\widetilde{k}$, using the above inequality \eqref{temp-keyineq0}, the convexity of $\mathbb{R}\ni t\mapsto t^{2}$ and  $\varepsilon_k\le\widetilde{\gamma}\widetilde{\varrho}^k$ leads to
  \begin{align*}
  \widetilde{\Lambda}_{k+k_1}\le\widetilde{\Lambda}_k&\le(1-\tau)\widetilde{\Lambda}_{k-1}+\tau\widehat{c}_2(\theta)\Big[\sum_{i=k-k_1}^{k+k_1}\!\Xi_{i-1}\!+\!\varepsilon_k\Big]^2\\
  &\le(1-\tau)\widetilde{\Lambda}_{k-1}+\tau(2k_1\!+\!2)\widehat{c}_2(\theta)\Big[\sum_{i=k-k_1}^{k+k_1}\!\Xi_{i-1}^2\!+\!\varepsilon_k^2\Big]\\
  &\le(1-\tau)\widetilde{\Lambda}_{k-k_1-1}+\tau(2k_1\!+\!2)\widehat{c}_2(\theta)\Big[\widetilde{\Lambda}_{k-k_1-1}-\widetilde{\Lambda}_{k+k_1}\!+\!\widetilde{\gamma}^2\widetilde{\varrho}^{2k}\Big]
  \end{align*}
  which, by setting  $\varrho_1\!:=\!\frac{1-\tau+\tau(2k_1+2)\widehat{c}_2(\theta)}{1+\tau(2k_1+2)\widehat{c}_2(\theta)}$, implies that 
  \begin{equation} \label{Eq-tLambda_k}
   \widetilde{\Lambda}_{k+k_1}\le\varrho_1\Big(\widetilde{\Lambda}_{k-k_1-1}+\widetilde{\gamma}^2\widetilde{\varrho}^{2k}\Big)
   = \varrho_1 \Lambda_{k-k_1-1}+\varrho_1\widetilde{\gamma}^2\widetilde{\varrho}^{2k}.
  \end{equation}
 Consequently, for any $k\ge\widetilde{k}\!+\!k_1$, it holds that 
 \begin{equation*}
 \widetilde{\Lambda}_k
  \le \varrho_1 \widetilde{\Lambda}_{k-2k_1-1} +\varrho_1 \widetilde{\gamma}^2\widetilde{\varrho}^{2k-2k_1}\le\varrho_1 \widetilde{\Lambda}_{k-2k_1-1} +\varrho_1 \widetilde{\gamma}^2\widetilde{\varrho}^{k-2k_1-1},
 \end{equation*}
 where the second inequality is due to $\widetilde{\varrho}<1$.
 Using the above recursive relation and letting $\overline{k}\!:=\!\lfloor\!\frac{k-\widetilde{k}-k_1}{2k_1+1}\!\rfloor$ and $\beta:=\varrho_1\widetilde{\gamma}^2$, we have
  \begin{equation*}
  \widetilde{\Lambda}_{k}
  \le\varrho_1^{\overline{k}}\widetilde{\Lambda}_{k-\overline{k}(2k_1+1)}  +\beta\widetilde{\varrho}^{k-2k_1-1}\Big[1+\frac{\varrho_1}{\widetilde{\varrho}^{2k_1+1}}+\cdots+\big(\frac{\varrho_1}{\widetilde{\varrho}^{2k_1+1}}\big)^{\overline{k}-1}\Big].
 \end{equation*}
 After a simple calculation for $\frac{\varrho_1}{\widetilde{\varrho}^{2k_1+1}}>1,\frac{\varrho_1}{\widetilde{\varrho}^{2k_1+1}}=1$ and $\frac{\varrho_1}{\widetilde{\varrho}^{2k_1+1}}<1$, respectively, there exist $\gamma>0$ and $\varrho\in(0,1)$ such that $\widetilde{\Lambda}_k\le\gamma\varrho^k$ for sufficiently large $k$.
	
 \noindent
 {\bf Case 2: $\theta\in(\frac{1}{2},1)$.} Since $\sum_{i=k\!-\!k_1}^{k+k_1}\!\!\Xi_i^2+\varepsilon_k^2<0.5$ for all $k\!\ge\!\widetilde{k}$ (if necessary by increasing $\widetilde{k}$), invoking inequality \eqref{temp-keyineq0} yields that for any $k\!\ge\!\widetilde{k}$,
\begin{align*}
 \widetilde{\Lambda}_{k+k_1}\le\widetilde{\Lambda}_k&\le(1-\tau)\widetilde{\Lambda}_{k-1}+\tau(2k_1+2)^{\frac{1}{2\theta}}\widehat{c}_2(\theta)\Big[\sum_{i=k-k_1}^{k+k_1}\Xi_{i-1}^2\!+\!\varepsilon_k^2\Big]^{\frac{1}{2\theta}}\\
  &\le(1-\tau)\widetilde{\Lambda}_{k-k_1-1}+\tau(2k_1+2)^{\frac{1}{2\theta}}\widehat{c}_2(\theta)\Big[\widetilde{\Lambda}_{k-k_1-1}-\widetilde{\Lambda}_{k+k_1}\!+\!\varepsilon_k^2\Big]^{\frac{1}{2\theta}}.
\end{align*}  
 By dividing the above inequality by $\tau$, after a suitable rearrangement, for any $k\!\ge\!\widetilde{k}$,
 \begin{align*}
 \widetilde{\Lambda}_{k+k_1}
  &\le\frac{(1-\tau)}{\tau}[\widetilde{\Lambda}_{k-k_1-1}-\widetilde{\Lambda}_{k+k_1}]+(2k_1+2)^{\frac{1}{2\theta}}\widehat{c}_2(\theta)\Big[\widetilde{\Lambda}_{k-k_1-1}-\widetilde{\Lambda}_{k+k_1}\!+\!\varepsilon_k^2\Big]^{\frac{1}{2\theta}}\\
&\le\!\tau^{-1}[1-\tau+\tau(2k_1+2)^{\frac{1}{2\theta}}\widehat{c}_2(\theta)]\Big((\widetilde{\Lambda}_{k-k_1-1}-\widetilde{\Lambda}_{k+k_1})
     \!+\!\widetilde{\gamma}^2k^{\frac{2\theta}{1-2\theta}}\Big)^{\frac{1}{2\theta}}.
 \end{align*}
 Note that $\frac{k}{k+k_1}\ge 0.5$ for all $k\ge\widetilde{k}$ (if necessary by increasing $\widetilde{k}$). Then, for any $k\!\ge\!\widetilde{k}$,
 \begin{align*}
 \widetilde{\Lambda}_{k+k_1}
  &\le\tau^{-1}[1\!-\!\tau\!+\tau(2k_1\!+\!2)^{\frac{1}{2\theta}}\widehat{c}_2(\theta)]\Big((\widetilde{\Lambda}_{k-k_1-1}\!-\!\widetilde{\Lambda}_{k+k_1})
     \!+\!2^{\frac{2\theta}{2\theta-1}}\widetilde{\gamma}^2(k+k_1)^{\frac{2\theta}{1-2\theta}}\Big)^{\frac{1}{2\theta}}\\
&\le M_1\max\big\{(\widetilde{\Lambda}_{k-k_1-1}-\widetilde{\Lambda}_{k+k_1})^{\frac{1}{2\theta}},
     (k+k_1)^{\frac{1}{1-2\theta}}\big\}
 \end{align*}
 with $M_1\!:=\tau^{-1}(1\!+\!2^{\frac{1}{2\theta-1}}\widetilde{\gamma}^{\frac{1}{\theta}})[1-\tau+\tau(2k_1\!+\!2)^{\frac{1}{2\theta}}\widehat{c}_2(\theta)]$. Thus, for all $k\ge\widetilde{k}\!+\!k_1$, we have
 \(
  \widetilde{\Lambda}_{k}\!\le M_1
  \!\max\!\big\{(\widetilde{\Lambda}_{k-2k_1-1}\!-\!\widetilde{\Lambda}_k)^{\frac{1}{2\theta}},k^{\frac{1}{1-2\theta}}\big\}.
  \)
 Invoking Lemma \ref{lemma-sequence} leads to the result.
 \end{proof}

 Now we are in a position to achieve the main conclusion of this section.
\begin{theorem}\label{KL-rate}
 Suppose that $\Phi$ has the KL property of exponent $\theta\in(0,1)$ at $\overline{x}$. If there exist $\widetilde{k}_0\in\mathbb{N},\widetilde{\gamma}>0$ and $\widetilde{\varrho}\in(0,1)$ such that inequality \eqref{rate-cond0} holds for all $k\ge\widetilde{k}_0$, then there exist $\gamma>0$ and $\varrho\in(0,1)$ such that for sufficiently large $k$,
 \begin{equation}\label{aim-ineq-rate}
 \|x^k-\overline{x}\|\le\Delta_k:=\sum_{j=k}^{\infty}\|x^{j+1}\!-\!x^j\|
  \le\left\{\begin{array}{cl}
  \gamma\varrho^{k} &{\rm if}\ \theta\in(0,1/2],\\
  \gamma k^{\frac{1-\theta}{1-2\theta}}&{\rm if}\ \theta\in(1/2,1).
 \end{array}\right.
\end{equation}
\end{theorem}
\begin{proof}
 By the definition of $\Delta_k$ and the triangle inequality, it holds $\|x^k-\overline{x}\|\!\le\!\Delta_k$. For each $k\in\mathbb{N}$, write $\Lambda_k\!:=\!\sum_{j=k}^{\infty}\Xi_j$. From \eqref{xk-bound} and $\ell(k)=k+m-1$, to prove the second inequality in \eqref{aim-ineq-rate}, it suffices to argue the existence of $\gamma>0$ and $\varrho\in(0,1)$ such that for sufficiently large $k$,
\begin{equation}\label{aim-ineq-rate1}
 \Lambda_{\ell(k)}
  \le\left\{\begin{array}{cl}
  \gamma\varrho^{k} &{\rm if}\ \theta\in(0,1/2],\\
  \gamma k^{\frac{1-\theta}{1-2\theta}}&{\rm if}\ \theta\in(1/2,1).
 \end{array}\right.
\end{equation}
 Let $\widetilde{k}$ be the same as before. Invoking the previous \eqref{temp-ineq32} with $\widehat{k}$ replaced by any $k\ge\widetilde{k}$ and $\varphi(t)=ct^{1-\theta}$ for some $c>0$, and passing the limit $\nu\to\infty$ results in
 \begin{align}\label{temp-keyineq}
 \Lambda_{\ell(k)}=\sum_{j=\ell(k)}^{\infty}\!\Xi_{j}
  & \le\!\frac{(1\!+\!2\sqrt{1\!-\!\tau})(2k_1\!+\!1)}{2\widehat{c}_1}\!\! \!\sum_{j=k-k_1}^{\ell(k)} \! \!\Xi_{j-1}\!  +\!\frac{c\widehat{c}}{\widehat{c}_1}\sum_{ j=k}^{\ell(k)} \! \!\big(C_j\!-\!\Phi(\overline{x})\big)^{1-\theta}\!+\!\frac{1}{\widehat{c}_1}\!\sum_{j=k}^{\infty}\varepsilon_j\nonumber\\
  &\le\!\frac{(1\!+\!2\sqrt{1\!-\!\tau})(2k_1\!+\!1)}{2\widehat{c}_1}\sum_{j=k-k_1}^{\ell(k)}\widetilde{\Lambda}_{j-1}^{\frac{1}{2}} \!+ \!\frac{c\widehat{c}}{\widehat{c}_1}
  \sum_{j=k}^{\ell(k)} \!\widetilde{\Lambda}_j^{1 \!-\theta}\!+\!\frac{1}{\widehat{c}_1}\sum_{j=k}^{\infty}\varepsilon_j\nonumber\\
  &\le
  \!\widetilde{c}_1\widetilde{\Lambda}_{k-k_1-1}^{\frac{1}{2}} \!+ \!\widetilde{c}_2
  \widetilde{\Lambda}_k^{1 \!-\theta}\!+\!\frac{1}{\widehat{c}_1}\sum_{j=k}^{\infty}\varepsilon_j
  \end{align}
  with 
  \(  \widetilde{c}_1\!:=\!\frac{(1\!+\!2\sqrt{1-\tau})(2k_1+1)(m+k_1)}{2\widehat{c}_1} 
  \) and \( \widetilde{c}_2\!:=\!\frac{c\widehat{c}m}{\widehat{c}_1},
  \)
  where the third inequality is due to the nonincreasing of the sequence $\{C_k\}_{k\in\mathbb{N}}$. 
  We proceed the proof by two cases.
	
  \noindent
  {\bf Case 1: $\theta\in(0,\frac{1}{2}]$.} Note that 
  $\lim_{k\to\infty}\widetilde{\Lambda}_k=0$, so $\widetilde{\Lambda}_{k-k_1-1}<\!0.5$ for all $k\ge\widetilde{k}$ (if necessary by increasing $\widetilde{k}$). From \eqref{rate-cond0}, for any $k\ge\widetilde{k}$, $\sum_{j=k}^{\infty}\varepsilon_j\le\widetilde{\gamma}\sum_{j=k}^{\infty}\widetilde{\varrho}^k\le\widetilde{\gamma}\frac{\widetilde{\varrho}^k}{1-\widetilde{\varrho}}$. Thus, together with $1-\theta\ge 1/2$ and inequality \eqref{temp-keyineq}, for any $k\ge\widetilde{k}$, it holds
  \[
  \Lambda_{\ell(k)}\le [\widetilde{c}_1+\widetilde{c}_2+\widehat{c}_1^{-1}]\Big[\widetilde{\Lambda}_{k-k_1-1}^{\frac{1}{2}}+\frac{\widetilde{\gamma}\widetilde{\varrho}^k}{1\!-\!\widetilde{\varrho}}\Big].
  \]
  By Lemma \ref{KL-fval-rate}, there exist $\widehat{\gamma}>0$ and $\widehat{\varrho}\in(0,1)$ such that $\widetilde{\Lambda}_k\le\widehat{\gamma}\widehat{\varrho}^k$ for sufficiently large $k$. Then, for all $k\ge\widetilde{k}$ (if necessary by increasing $\widetilde{k}$), 
  \[  \Lambda_{\ell(k)}\le[\widetilde{c}_1+\widetilde{c}_2+\widehat{c}_1^{-1}]\Big(\widehat{\gamma}^{\frac{1}{2}}\widehat{\varrho}^{\frac{k-k_1-1}{2}}+\frac{\widetilde{\gamma}\widetilde{\varrho}^k}{1\!-\!\widetilde{\varrho}}\Big).
  \]
  Set $\widetilde{c}_3\!:=\!(\widetilde{c}_1+\widetilde{c}_2+\widehat{c}_1^{-1})(\gamma^{\frac{1}{2}}\widehat{\varrho}^{-\frac{k_1+1}{2}}+\frac{\widetilde{\gamma}}{1-\widetilde{\varrho}})$ and $\varrho_1:=\max\{\widehat{\varrho}^{\frac{1}{2}},\widetilde{\varrho}\}$. The above inequality implies $\Lambda_{\ell(k)}\le\widetilde{c}_3\varrho_1^k$ for all $k\ge\widetilde{k}$.
	
 \noindent
 {\bf Case 2: $\theta\in(\frac{1}{2},1)$.} Since $\widetilde{\Lambda}_{k-k_1-1}<0.5$ for all $k\ge\widetilde{k}$ (if necessary by increasing $\widetilde{k}$), 
 from inequality \eqref{temp-keyineq}, it follows that for any $k\ge\widetilde{k}$, 
 \begin{equation}\label{temp-Lambdak}
  \Lambda_{\ell(k)}\le [\widetilde{c}_1+\widetilde{c}_2+\widehat{c}_1^{-1}]\big[\widetilde{\Lambda}_{k-k_1-1}^{1-\theta}+\textstyle{\sum_{j=k}^{\infty}}\varepsilon_j\big]. 
  \end{equation}
  By Lemma \ref{KL-fval-rate}, there exists $\widehat{\gamma}>0$ such that $\widetilde{\Lambda}_k\le\widehat{\gamma}k^{\frac{1}{1-2\theta}}$ for sufficiently large $k$, so $\widetilde{\Lambda}_{k-k_1-1}\le\widehat{\gamma}(k-k_1-1)^{\frac{1}{1-2\theta}}$ for all $k\ge\widetilde{k}$ (if necessary by enlarging $\widetilde{k}$). Note that $\frac{k-k_1-1}{k}>0.5$ for any $k\ge\widetilde{k}$ (if necessary by enlarging $\widetilde{k}$). Then, for any $k\ge\widetilde{k}$, $\widetilde{\Lambda}_{k-k_1-1}\le 2^{\frac{1}{2\theta-1}}\widehat{\gamma}k^{\frac{1}{1-2\theta}}$. In addition, from $\varepsilon_k\le\widetilde{\gamma}k^{\frac{\theta}{1-2\theta}}$, it is not hard to get 
 \[
 \sum_{j=k}^{\infty}\varepsilon_j\le\widetilde{\gamma}\sum_{j=k}^{\infty}k^{\frac{\theta}{1-2\theta}}\le\widetilde{\gamma}\int_{k}^{\infty}t^{\frac{\theta}{1-2\theta}}dt=\frac{\widetilde{\gamma}(2\theta-1)}{1-\theta}k^{\frac{1-\theta}{1-2\theta}}.
 \]  
 Together with $\widetilde{\Lambda}_{k-k_1-1}\le 2^{\frac{1}{2\theta-1}}\widehat{\gamma}k^{\frac{1}{1-2\theta}}$ and inequality \eqref{temp-Lambdak}, for any $k\ge\widetilde{k}$, it holds
  \begin{align*}
  \Lambda_{\ell(k)}  \le[\widetilde{c}_1+\widetilde{c}_2+\widehat{c}_1^{-1}]\Big(2^{\frac{1-\theta}{2\theta-1}}\widehat{\gamma}^{1-\theta}k^{\frac{1-\theta}{1-2\theta}}+\frac{\widetilde{\gamma}(2\theta-1)}{1-\theta}k^{\frac{1-\theta}{1-2\theta}}\Big).
  \end{align*}
  Set $\widetilde{c}_4\!:=\!(\widetilde{c}_1\!+\widetilde{c}_2+\widehat{c}_1^{-1}\!)(2^{\frac{1-\theta}{2\theta-1}}\widehat{\gamma}^{1-\theta}\!+\frac{\widetilde{\gamma}(2\theta-1)}{1-\theta})$. We get $\Lambda_{\ell(k)}\le\widetilde{c}_4k^{\frac{1-\theta}{1-2\theta}}$ for all $k\ge\widetilde{k}$. 
 \end{proof}
\begin{remark}\label{remark-result}
 Theorems \ref{KL-converge} and \ref{KL-rate} extend the convergence results obtained in \cite{Attouch09,Attouch13,Frankel15} for a monotone iterative framework to a nonmonotone case. Compared with the nonmonotone iterative framework proposed in \cite{QianPan23} by 
 GLL-type nonmonotone line-search procedure, the ZH-type iterative framework has better convergence, i.e., the full convergence (local convergence rate) of its iterate sequence does not require additional restriction on $\Phi$ except its KL property (KL property of exponent $\theta\in(0,1)$). In contrast, to obtain the global convergence, the paper \cite{QianPan23} requires an assumption on the growth of an objective value subsequence, which is shown to hold automatically if $\Phi$ is also $\rho$-weakly convex on a neighborhood of the set of critical points. 
 We also note that the favourable properties of the sequence $\{C^k\}_{k\in\mathbb{N}}$ plays a crucial role in the convergence proof of the iterate sequence $\{x^k\}_{k\in\mathbb{N}}$. 
\end{remark} 

To close this section, we would like to make an important point that the proofs of this section still hold if we modify condition H1. Those variants of H1 were used in some existing research. Therefore, the obtained results may be used to derive new convergence results for some existing algorithms not covered by H1-H3. 
We briefly discuss two variants of H1 below. The first variant of H1 is the following one:
\begin{itemize}
\item [{\bf H1a.}] For each $k\in\mathbb{N}$, $\Phi(x^{k})+a_{k}\|x^{k}\!-\!x^{k-1}\|^2\le
	\Phi(x^{k-1})+\nu_{k-1}$, where $0\le\nu_{k}\le(1\!-\!\tau_{k})[\Phi(x^{k-1})+\nu_{k-1}-\Phi(x^{k})]$ with $\nu_0=0$ and $\tau_{k}\in[\tau,1]$ for some $\tau\in(0,1]$.
\end{itemize}
\vspace{0.2cm}
\noindent
By letting $C_{k-1}\!\!:=\!\Phi(x^{k-1})\!+\!\nu_{k-1}$ and using $\nu_k\le(1\!-\!\tau_{k})[\Phi(x^{k-1})\!+\!\nu_{k-1}\!-\!\Phi(x^{k})]$, it holds $C_{k}\!\le\!(1\!\!-\!\tau_{k})C_{k-1}\!\!+\!\tau_k\Phi(x^k)$. Then, using the same arguments as before can prove that any sequence satisfying conditions H1a and H2-H3 also has the convergence results of Theorems \ref{KL-converge} and \ref{KL-rate}. One can check that the iterate sequence of \cite[Algorithm 1]{Grapiglia17} satisfies conditions H1a and H2-H3 under Assumptions A1-A3 there.

The second variant of condition H1 is replaced by the following procedure:
\begin{itemize}
	\item [{\bf H1b.}] For each $k\!\in\!\mathbb{N}$, $\Phi(x^{k})\!+\!a_{k}\!\|x^k\!\!-\!x^{k\!-\!1}\|^2\!\!\le\! C_{k\!-\!1}\!\!+\!\eta_{k\!-\!1}^2$ where $a_k\!\!\ge\!\!\underline{a}$ for some $\underline{a}\!\!>\!\!0$ and
	$C_{k}=(1\!\!-\!\tau_{k})C_{k\!-\!1}+\!\tau_{k}\Phi(x^{k})$ with $C_0\!=\!\Phi(x^0)$ and $\tau_{k}\!\in\![\tau,1]$ for a $\tau\!\in\!(0,1]$.
\end{itemize}
\noindent
Sun \cite{Sun21} achieved the convergence of the iterate sequence satisfying conditions H1b with $\tau_k\equiv1$ and H2-H3, by assuming that $\Phi$ is a KL function and $\{\eta_k\}_{k\in\mathbb{N}}$ satisfies
\begin{align}\label{etak-ass}
	\{\eta_k\}_{k\in\mathbb{N}}\subset\mathbb{R}_{+}\ {\rm with}\  \sum_{k=1}^{\infty}\eta_k<\infty\ \  
	{\rm and}\ \ \sum_{k=1}^{\infty}\Big(\sum_{l=k}^{\infty}\eta_l^2\Big)^{\frac{\varsigma-1}{\varsigma}}<\infty\ {\rm for\ some}\ \varsigma>1.
\end{align}
In fact, for any sequence $\{x^k\}_{k\in\mathbb{N}}$ obeying conditions H1b and H2-H3, by introducing the potential function
$\Psi(z):=\Phi(x)+\varsigma^{-1}t^{\varsigma}$ for $z=(x,t)\in\mathbb{X}\times\mathbb{R}_{++}$ and following the similar arguments to those for Theorem \ref{KL-converge}, we can obtain the same convergence result under the assumption that $\Phi$ is a KL function and $\{\eta_k\}_{k\in\mathbb{N}}$ satisfies \eqref{etak-ass}.

 \section{Applications of the iterative framework}\label{sec4}

This section demonstrates that the novel iterative framework encompasses
some existing algorithms. We focus on two particular examples: Proximal gradient method (PGM) and Riemannian gradient method (RGM).  
We first show that the iterate sequence of a PGM with the ZH-type nonmonotone line-search for nonconvex and nonsmooth composite problems comply with conditions H1-H3.
We then apply the conclusions of Theorems \ref{KL-converge} and \ref{KL-rate} to provide its full convergence certificate under the KL property of objective functions. To the best of our knowledge, this convergence result is new.
 \subsection{Convergence of PGM with ZH-type nonmonotone line-search}\label{sec4.1}

 \ \\
 Consider the following nonconvex and nonsmooth composite optimization problem 
 \begin{equation}\label{Fprob}
 \min_{x\in\mathbb{X}}\Theta(x):=f(x)+g(x),
 \end{equation}
 where $f\!:\mathbb{X}\to\overline{\mathbb{R}}$ and  $g\!:\mathbb{X}\to\overline{\mathbb{R}}$ satisfy the following basic assumption: 
 \begin{aassumption}\label{ass0}
 \begin{description}
 \item[(i)] $g$ is a proper, lsc, and prox-bounded function with ${\rm dom}\,g\ne\emptyset$ (for the prox-boundedness, the reader may refer to \cite[Definition 1.23]{RW98});

 \item[(ii)] $f$ is differentiable on an open set $\mathcal{O}\!\supset{\rm dom}\,g$ with $\nabla\!f$ strictly continuous on $\mathcal{O}$; 

 \item[(iii)] the function $\Theta$ is lower bounded, i.e., $\inf_{x\in\mathbb{X}}\Theta(x)>-\infty$.
 \end{description}
 \end{aassumption}

 Recently, Marchi \cite{Marchi23} proposed a PGM with the ZH-type nonmonotone line-search strategy for solving \eqref{Fprob}, and its iteration steps are described as follows, where 
 \[
 \Prox_{\gamma g} (z) :=\mathop{\arg\min}_{x\in\mathbb{X}}\left\{ \frac{1}{2\gamma}\|x-z\|^2+g(x) 
 \right\}
 \]
 is the proximal mapping of the function $g$ associated with the parameter $\gamma>0$. 
 \begin{algorithm}[H]
 \caption{\label{PGMnls}{\bf\,(Nonmonotone line-search PGM)}}
 \textbf{Initialization:} Select $0<\!\gamma_{\rm min}\le\!\gamma_{\rm max}<\infty,\,\alpha,\beta\in(0,1),\,p_{\min}\in(0,1]$.
 Choose $x^0\in{\rm dom}g$. Set $C_0=\Theta(x^0)$ and $k:=0$.\\
 \textbf{while} the termination condition is not satisfied \textbf{do}
 \begin{enumerate}
  \item  Choose $\gamma_{k,0}\in[\gamma_{\rm min},\gamma_{\rm max}]$.

  \item  \textbf{For} $l=0,1,2,\ldots$ \textbf{do}

  \item \quad Let $\gamma_k=\gamma_{k,0}\beta^{l}$ and compute 
  $x^{k+1}\in \Prox_{\gamma_k g} (x^k\!-\!\gamma_k\nabla\!f(x^k))$.
  

  \item \quad If $\Theta(x^{k+1})\le C_k-\frac{\alpha}{2\gamma_k}\|x^{k+1}-x^k\|^2$, then go to Step 6.

  \item  \textbf{end for}

  \item Choose $p_k\in[p_{\min},1]$, and set 
        $C_{k+1}=(1-p_k)C_k+p_k\Theta(x^{k+1})$.

  \item  Set $k\leftarrow k+1$ and go to Step 1.
 \end{enumerate}
 \textbf{end (while)}
 \end{algorithm}

 From \cite[Lemma 4.1]{Marchi23}, Algorithm \ref{PGMnls} is well defined, and for its iterate sequence, Marchi \cite{Marchi23} only achieved the subsequence convergence. Next we prove that the iterate sequence of Algorithm \ref{PGMnls} meets conditions H1-H3.
 Consequently, its full convergence and local convergence rate
 naturally follows the convergence results of Section \ref{sec3}. 
 
 \begin{theorem}\label{theorem-PGels}
  Suppose the sequence $\{x^k\}_{k\in\mathbb{N}}$ of Algorithm \ref{PGMnls} is bounded. Then, 
  \begin{itemize}
  \item [(i)] when $\Theta$ is a KL function, $\sum_{k=0}^{\infty}\|x^k\!-\!x^{k-1}\|<\infty$ and $\{x^k\}_{k\in\mathbb{N}}$ is convergent;

  \item[(ii)] when $\Theta$ is a KL function of exponent $\theta\in[1/2,1)$, $\{x^k\}_{k\in\mathbb{N}}$ converges to a point $\widetilde{x}\in{\rm crit}\,\Theta$ and there exist $\overline{k}\in\mathbb{N},\gamma>0$ and $\varrho\in(0,1)$ such that 
              \begin{equation*}
                 \|x^k-\widetilde{x}\|\le\sum_{j=k}^{\infty}\|x^{j+1}\!-\!x^j\|
                 \le\left\{\begin{array}{cl}
                 \gamma\varrho^{k} &{\rm if}\ \theta=1/2,\\
                 \gamma k^{\frac{1-\theta}{1-2\theta}}&{\rm if}\ \theta\in(1/2,1)
                 \end{array}\right.\ {\rm for\ all}\ k\ge\overline{k}.
              \end{equation*}
  \end{itemize}
 \end{theorem}
 \begin{proof}
  From the proof of \cite[Lemma 4.1]{Marchi23}, it follows that for every $k\in\mathbb{N}$,
 \begin{equation}\label{descent-Theta}
  \Theta(x^{k+1})\le C_k-\frac{\alpha}{2\gamma_k}\|x^{k+1}-x^k\|^2.
 \end{equation}
 In addition, from the boundedness of $\{x^k\}_{k\in\mathbb{N}}$ and \cite[Corollary 4.5]{Marchi23}, the argument by contradiction shows that there exists $\underline{\gamma}>0$ such that $\gamma_k\ge\underline{\gamma}$ for all $k\in\mathbb{N}$. Thus, the sequence $\{x^k\}_{k\in\mathbb{N}}$ satisfies condition H1 with $a_k\!\equiv\!\frac{\alpha}{2\underline{\gamma}}$. From Lemma~\ref{lemma1-Phi}(iii), $\lim_{k\to\infty}(x^{k}-x^{k-1})=0$. As the sequence $\{x^k\}_{k\in\mathbb{N}}$ is bounded, there exists a convergent subsequence $\{x^{k_j}\}_{j\in\mathbb{N}}$ with the limit, say $\overline{x}$,  i.e., $\lim_{j\to\infty}x^{k_j}=\overline{x}$. Together with $\lim_{k\to\infty}(x^{k}-x^{k-1})=0$, we obtain $\lim_{k\to\infty}x^{k_j-1}=\overline{x}$. We next argue that $\limsup_{j\to\infty}\Theta(x^{k_j})\le\Theta(\overline{x})$. 
 From the definition of $x^{k+1}$, for each $j\in\mathbb{N}$, 
 \begin{align*}
 &\langle \nabla\!f(x^{k_j-1}),x^{k_j}-x^{k_j-1}\rangle+\frac{1}{2\gamma_{k_j-1}}\|x^{k_j}-x^{k_j-1}\|^2+\Theta(x^{k_j})\\
 &\le\langle \nabla\!f(x^{k_j-1}),\overline{x}-x^{k_j-1}\rangle+\frac{1}{2\gamma_{k_j-1}}\|\overline{x}-x^{k_j-1}\|^2+\Theta(\overline{x})+f(x^{k_j})-f(\overline{x}).
 \end{align*}
 Passing the limit $j\to\infty$ and using $\lim_{j\to\infty}x^{k_j}=\overline{x}=\lim_{j\to\infty}x^{k_j-1}$ and $\gamma_{k_j-1}\ge\underline{\gamma}$, we conclude that  $\limsup_{j\to\infty}\Theta(x^{k_j})\le\Theta(\overline{x})$, so condition H3 also holds.  
  
 Denote $\omega(x^0)$ by the set of accumulation points of $\{x^k\}_{k\in\mathbb{N}}$. We next prove that condition H2 holds. For each $k\in\mathbb{N}$, from the definition of $x^{k+1}$ in step 3, 
  \[
   0\in\nabla\!f(x^{k})
	+\frac{1}{\gamma_k}(x^{k+1}\!-\!x^k)+\partial g(x^{k+1}),
  \]
  which together with the expression of $\Theta$ implies that for each $k\in\mathbb{N}$,
  \begin{equation}\label{aim-ineq2}
   \partial\Theta(x^{k+1})\ni w^{k+1}\!:=\nabla\!f(x^{k+1})-\nabla f(x^k)-\frac{1}{\gamma_k}\big(x^{k+1}-x^k\big).
  \end{equation}
  We claim that there exists $\overline{L}>0$ such that for all $k\in\mathcal{K}:=\{k\in\mathbb{N}\ |\ x^{k+1}\ne x^k\}$,
  \begin{equation}\label{aim-ineq3}
   \limsup_{\mathcal{K}\in k\to\infty}\frac{\|\nabla\!f(x^{k+1})\!-\!\nabla\!f(x^{k})\|}{\|x^{k+1}\!-\!x^k\|}\le \overline{L}.
  \end{equation}
  If not, there must exist an index set $\widehat{\mathcal{K}}\subset\mathcal{K}$ such that
  \begin{equation}\label{temp-ineq}
  \lim_{\widehat{\mathcal{K}}\ni k\to\infty}\frac{\|\nabla\!f(x^{k+1})\!-\!\nabla\!f(x^{k})\|}{\|x^{k+1}\!-\!x^k\|}=\infty.
  \end{equation}
  We assume that $\lim_{\widehat{\mathcal{K}}\ni k\to\infty}x^k=\widehat{x}^*$ (if necessary by  taking a subsequence of $\{x^k\}_{k\in\widehat{\mathcal{K}}}$). Clearly, $\widehat{x}^*\in\omega(x^0)$. By Lemma~\ref{lemma1-Phi}(iii), $\lim_{\widehat{\mathcal{K}}\ni k\to\infty}\big[x^{k}+(x^{k+1}\!-\!x^{k})\big]=\widehat{x}^*$. Since $\nabla\!f$ is strictly continuous at $\widehat{x}^*$, there exists $\widehat{L}>0$ such that for all $k\in\widehat{\mathcal{K}}$ large enough,
  \[
    \|\nabla\!f(x^{k+1})-\nabla\!f(x^{k})\|
    \le\widehat{L}\|x^{k+1}-x^k\|,
   \]
  which is a contradiction to \eqref{temp-ineq}. Consequently, the claimed inequality \eqref{aim-ineq3} holds. Thus, for all $k\in\mathcal{K}$,
  $\|\nabla\!f(x^{k+1})\!-\!\nabla\!f(x^{k})\|
   \le \overline{L}\|x^{k+1}\!-\!x^k\|$.
  Together with the definition of $w^{k+1}$, it follows that $\|w^{k+1}\|\le(\frac{1}{\gamma_k}\!+\!\overline{L})\|x^{k+1}-x^{k}\|$ for all $k\in\mathcal{K}$. Recall that for each $k\in\mathbb{N}$, $w^{k+1}\!\in\partial\Theta(x^{k+1})$ and $\gamma_k\ge\underline{\gamma}$, so $\|w^{k+1}\|\le b \|x^{k+1}-x^{k}\|$ for each $k\in\mathbb{N}$ with  $b=\underline{\gamma}^{-1}+\overline{L}>0$. This shows that condition H2 holds. The desired conclusions follow Theorems \ref{KL-converge} and \ref{KL-rate} with $\widetilde{x}=\overline{x}\in(\partial\Phi)^{-1}(0)$. 
 \end{proof}

 \subsection{Convergence of RGM with ZH-type nonmonotone line-search}\label{sec4.2}

  Let $\mathcal{M}$ be an embedded submanifold of the finite-dimensional vector space $\mathbb{X}$. Consider 
  \begin{align}\label{maniprob}
  \min_{x\in\mathbb{X}}F(x):=f(x)+\delta_{\mathcal{M}}(x),
  \end{align}
  where $f\!:\mathcal{O}\supset\mathcal{M}\to\mathbb{R}$ is an $L_{\!f}$-smooth function and $\mathcal{O}$ is an open set of $\mathbb{X}$, and $\delta_{\mathcal{M}}$ denotes the indicator function of $\mathcal{M}$. Assume that the function $F$ is lower bounded, which automatically holds if $\mathcal{M}$ is compact. For this problem, Wen and Yin \cite{Wen13} and Oviedo \cite{Oviedo22} proposed the following Riemannian gradient method (RGM) with the ZH-type nonmonotone line-search strategy, where ${\rm grad}f(x^k)$ denotes the Riemannian gradient of $f$ at $x^k$, ${\rm T}_{x^k}\mathcal{M}$ is the tangent space of $\mathcal{M}$ at $x^k$, ${\rm R}_{x^k}(\cdot)$ denotes the retraction mapping at $x^k$ from the tangent bundle to $\mathcal{M}$ (see \cite[Definition 2.1]{Boumal18}).
 \begin{algorithm}[h]
 \caption{\label{ManPGMnls}{\bf\,(Nonmonotone line-search RGM)}}
 \textbf{Initialization:} Select $0<\!\alpha_m\le\alpha_M<\!\infty,\,\rho_1\in(0,1),\rho_2\!\in[0,1),\beta\in(0,1)$ and $p_{\min}\in(0,1]$.
 Choose $x^0\in\!\mathcal{M}$ and $z^0\in{\rm T}_{x^0}\mathcal{M}$ satisfying \eqref{dir-equa}. Set $C_0=F(x^0)$ and $k:=0$.\\
 \textbf{while} the termination condition is not satisfied \textbf{do}
 \begin{enumerate}
  \item  Choose $\alpha_{k,0}\in[\alpha_{m},\alpha_{M}]$.

  \item  \textbf{For} $l=0,1,2,\ldots$ \textbf{do}

  \item \quad Set $\alpha_k=\beta^l\alpha_{k,0}$.

  \item \quad If $f({\rm R}_{x^k}(\alpha_kz^k))\le C_k+\rho_1\alpha_k\langle {\rm grad}f(x^k),z^k\rangle-\rho_2\alpha_k^2\|z^k\|^2$, go to step 6.

  \item  \textbf{end for}

  \item Let $x^{k+1}={\rm R}_{x^k}(\alpha_kz^k)$ and select a direction $z^{k+1}\in{\rm T}_{x^{k+1}}\mathcal{M}$ satisfying \eqref{dir-equa}. 
  
  \item Choose $p_k\in[p_{\min},1]$, and set 
        $C_{k+1}=(1-p_k)C_k+p_kf(x^{k+1})$.

  \item  Set $k\leftarrow k+1$ and go to step 1.
 \end{enumerate}
 \textbf{end (while)}
 \end{algorithm}

 For Algorithm \ref{ManPGMnls}, as far as we know, only the gradient sequence $\{{\rm grad}f(x^k)\}_{k\in\mathbb{N}}$ is proved to converge to zero under the following condition with $c_1>0$ and $c_2>0$:
 \begin{equation}\label{dir-equa}
  \langle {\rm grad}f(x^k),z^k\rangle\le -c_1\|{\rm grad}f(x^k)\|^2\ \ {\rm and}\ \ \|z^k\|\le c_2\|{\rm grad}f(x^k)\|\quad\forall k\in\mathbb{N}.
  \end{equation} 
 Next we establish the full convergence and local convergence rate of the iterate sequence by arguing that it conforms to conditions H1-H3. Consequently, the convergence results in Section~\ref{sec3} are applicable to it. 
 \begin{theorem}\label{theorem-ManPGels}
  Suppose that the submanifold $\mathcal{M}$ is compact, and that $F$ is a KL function. Then, for the sequence $\{x^k\}_{k\in\mathbb{N}}$ generated by Algorithm \ref{ManPGMnls}, it holds that $\sum_{k=0}^{\infty}\|x^k\!-\!x^{k-1}\|<\infty$. If in addition $F$ is a KL function of exponent $\theta\in[\frac{1}{2},1)$, then it converges to some $\widetilde{x}\in{\rm crit}\,F$ and there exist $\overline{k}\in\mathbb{N},\gamma>0$ and $\varrho\in(0,1)$ such that
  \begin{equation*}
  \|x^k-\widetilde{x}\|\le\sum_{j=k}^{\infty}\|x^{j+1}\!-\!x^j\|
  \le\left\{\begin{array}{cl}
    \gamma\varrho^{k} &{\rm if}\ \theta=1/2,\\
     \gamma k^{\frac{1-\theta}{1-2\theta}}&{\rm if}\ \theta\in(1/2,1)
    \end{array}\right.\ \ {\rm for\ all}\ k\ge\overline{k}.
  \end{equation*}
 \end{theorem}
 \begin{proof} 
Since $\{x^k\}_{k\in\mathbb{N}}\!\subset\mathcal{M}$, the compactness of $\mathcal{M}$ implies the boundedness of $\{x^k\}_{k\in\mathbb{N}}$. Then, there exists a subsequence $\{x^{k_j}\}_{j\in\mathbb{N}}$ with $\lim_{j\to\infty}x^{k_j}=\overline{x}\in\mathcal{M}$ such that $\lim_{j\to\infty}F(x^{k_j})=F(\overline{x})$, and condition H3 holds. By combining \cite[Eq. (B.3)]{Boumal18} with $x^{k+1}={\rm R}_{x^k}(\alpha_kz^k)$, there exists a constant $\widetilde{\alpha}>0$ such that for each $k\!\in\!\mathbb{N}$, $\|x^{k+1}\!-\!x^k\|^2\!\le \!\widetilde{\alpha}^2\alpha_{M}^2\|z^k\|^2$. From step 4 of Algorithm \ref{ManPGMnls} and \eqref{dir-equa}, for each $k\!\in\!\mathbb{N}$, 
 \[   
 F(x^{k+1})\!=\!f(x^{k+1})\!\le\! C_k\!-\!\rho_1c_1\alpha_k\|{\rm grad}f(x^k)\|^2\!-\!\rho_2\alpha_k^2\|z^k\|^2\!\le\! C_k\!-\!\frac{\rho_1c_1}{c_2}\alpha_k\|z^k\|^2.
 \]
 In addition, from \cite[Lemma 2]{Oviedo22}, there exist $k_0\in\mathbb{N}$ and $\underline{\alpha}>0$ such that $\alpha_k\ge\underline{\alpha}$ for all $k\ge k_0$.
 These two facts show that condition H1 holds. 
 The remaining part is to show
 that condition H2 holds. By Lemma~\ref{lemma1-Phi}(i)-(ii), the sequences $\{C_k\}_{k\in\mathbb{N}}$ and $\{F(x^k)\}_{k\in\mathbb{N}}$ converge to the same limit. Together with step 4 of Algorithm \ref{ManPGMnls}, we have $\lim_{k\to\infty}\alpha_k\|z^k\|^2=0$, which along with $\alpha_k\ge\underline{\alpha}$ for all $k\in\mathbb{N}$ implies that $\lim_{k\to\infty}\|z^k\|=0$. From $\alpha_k\le\alpha_{\rm max}$ for all $k$ and \cite[Eq. (B.4)]{Boumal18}, we have $\|{\rm R}_{x^k}(\alpha_kz^k)\!-\!(x^k\!+\!\alpha_k z^k)\|\!=\!o(\|\alpha_kz^k\|)$. If necessary by increasing $k_0$, for each $k\ge k_0$, 
 \begin{equation}\label{gradf-equa}
  \|x^{k+1}\!-\!x^k\|\ge \alpha_k\|z^k\|\!-\!\|{\rm R}_{x^k}(\alpha_kz^k)\!-\!(x^k\!+\!\alpha_k z^k)\|\!\ge\!
  \frac{1}{2}\underline{\alpha}\|z^k\|\ge\frac{1}{2}\underline{\alpha}c_1\|{\rm grad}f(x^k)\|, 
  \end{equation}
  where the third inequality is obtained by using the first inequality in \eqref{dir-equa}.
  Note that $\partial F(x^k)=\nabla f(x^k)+{\rm N}_{x^k}\mathcal{M}$ for each $k\in\mathbb{N}$, where ${\rm N}_{x^k}\mathcal{M}$ denotes the normal space of $\mathcal{M}$ at $x^k$. Then,
  we have for each $k\in\mathbb{N}$,
 \[  
 {\rm dist}(0,\partial F(x^k))=\|{\rm Proj}_{{\rm T}_{\!x^k}\mathcal{M}}(\nabla\!f(x^k))\|=\|{\rm grad}\,f(x^k)\|\le 2/(\underline{\alpha}c_1)\|x^{k+1}-x^k\|,
 \]
 where the inequality is due to \eqref{gradf-equa}. This shows that  condition H2 holds with $k_1=1$ and $\varepsilon_k=0$ for all $k\ge k_0$. Now the desired conclusions follow Theorems \ref{KL-converge} and \ref{KL-rate}. 
 \end{proof}

 From \cite[Lemma 2.10]{QianPanXiao24}, the KL property of $F$ at $x\in\mathcal{M}$ in terms of Definition \ref{KL-def} is equivalent to the KL property of $f|_{\mathcal{M}}$, the restriction of $f$ on $\mathcal{M}$, at $x\in\mathcal{M}$ in terms of \cite[Definition 3.5]{Cruz2013}. 

 \section{Conclusion}\label{sec5}

In this paper, we proposed a novel nonmonotone descent iterative framework consisting of the ZH-type nonmonotone decrease condition and a relative error condition.
We proved that any iterative sequence complying with this framework enjoys full convergence when $\Phi$ is a KL function, and the convergence is linear if $\Phi$ is a KL function of exponent $1/2$. 
This answers the question whether a descent method with ZH-type nonmonotone 
line search strategy converges.
We also demonstrated the cases of existing algorithms that fall into
the proposed framework. As a result, new convergence results are readily
available for those algorithms as a consequence of our obtained results.
Furthermore, the proofs of the main results may be adapted to modified frameworks that include more existing algorithms as special instances. 
This shows the potential of the proposed framework to obtain new 
convergence results for existing algorithms.
Compared with the nonmonotone iterative framework \cite{QianPan23}, the new one possesses the full convergence (local convergence rate) without additional restriction on $\Phi$ except its KL property (KL property of exponent $\theta\in(0,1)$). 
We hope to explore more applications of the proposed framework in our next research project.
  

\section*{Appendix: ZH-type nonmonotone line search scheme satisfies (H1)-(H3)}\label{secA}

We first recall the nonmonotone line-search algorithm (NLSA) proposed in \cite{ZhangH04}.
Let $\Phi=f$, which is a continuously differentiable function and is bounded from below on $\mathbb{X}=\mathbb{R}^n$. 
We follow the notation used in \cite{ZhangH04}. In particular,
NLSA uses  
$x_k$ with subscript $k$ for its iterates. 
Furthermore, it uses
$\bfg_k := \nabla f(x_k)$ (gradient of $f$ at $x_k$) and $\bfd_k$ for the search direction  at $x_k$.
The {\em Direction Assumption} used is \cite[(2.4), (2.5)]{ZhangH04}:
\begin{equation} \label{Direction-Assumption}
	 \bfg_k^\top \bfd_k \le - c_1 \| \bfg_k\|^2 \quad
	 \mbox{and} \quad
	 \| \bfd_k\| \le c_2 \| \bfg_k\| ,
\end{equation} 
for some positive constants $c_1$ and $c_2$.
NLSA uses nonmonotone Wolfe conditions or nonmonotone Armijo conditions to select its steplength
$\alpha_k$. To simplify our validation, we use the latter for the demonstration.
The Armijo search inequality on $\alpha_k$ is \cite[(1.4)]{ZhangH04}:
\begin{equation} \label{ZH-Linesearch}
   f(x_k + \alpha_k \bfd_k) \le C_k + \delta \alpha_k \bfg_k^\top \bfd_k \quad \mbox{and} \quad
   \alpha_k \le \mu,
\end{equation} 
where $0< \delta <1$, $\mu>0$, and $C_k$ is updated as follows \cite[(1.6)]{ZhangH04}:
\begin{equation}\label{ZH-QC}
	 Q_{k+1} = \eta_k Q_k + 1, \
	 C_{k+1} = (\eta_k Q_k C_k + f(x_{k+1}))/Q_{k+1}, \
	 C_0 = f(x_0)  \ \mbox{and} \ Q_0 = 1.
\end{equation} 
Here, $\eta_k \in [\eta_{\min}, \eta_{\max}]$ with $0 \le \eta_{\min} \le \eta_{\max} \le 1$. The update on the iterate is $x_{k+1} = x_k + \alpha_k \bfd_k$.

With the above setting, \cite[Lemma~2.1]{ZhangH04} shows that there exists $\underline{\alpha} >0$ such that $\alpha_k \ge \underline{\alpha}$. We now prove all three conditions in (H1)-(H3)
are satisfied with the sequence $\{x_k\}$.

{\bf On H1.} It follows from \cite[(2.8)]{ZhangH04} that for some $\beta>0$
\begin{align*}
	f(x_{k+1}) &\le C_k - \beta \| \bfg_k\|^2   \quad \mbox{(by \cite[(2.8)]{ZhangH04})}\\
	&\le C_k - \frac{\beta}{c_2^2} \| \bfd_k\|^2 \quad \mbox{(by \eqref{Direction-Assumption})}\\
	&= C_k - \frac{\beta}{\alpha_k^2 c_2^2}  \| x_{k+1} - x_k\|^2 .
\end{align*}
Referring to \eqref{ZH-QC},
H1 holds with the following choices of $a_k$ and $\tau_k$. Let
 \[
 a_k :=\beta/( \alpha_{k-1}^2 c_2^2) \stackrel{\eqref{ZH-Linesearch}}{\ge} \beta/( \mu^2c_2^2) =: \underline{a} >0
\] and 
\[
\tau_k := \frac{1}{\eta_{k-1}Q_{k-1}\!+1}\ge 1-\eta_{\rm max} :=\tau >0,
\]
where the third inequality used the fact $Q_k \le 1/(1- \eta_{\max})$ established in \cite[(2.15)]{ZhangH04} when $\eta_{\max} < 1$.

{\bf On H2.} According to the direction assumption \eqref{Direction-Assumption}, for sufficiently large $k$,
\[
 \|g_k\|^2\le -c_1^{-1}g_k^{\top}d_k\le c_1^{-1}\|g_k\|\|d_k\|=(\alpha_kc_1)^{-1}
 \|g_k\|\|x_{k+1}-x_k\|,
 \]
 which implies that condition H2 holds with $k_1=1,b_k = c_1\alpha_k$ for sufficiently large $k$ and $\varepsilon_k\equiv 0$. Since $\alpha_k \ge \underline{\alpha}$, we must have
 $\sum_{k=1}^{\infty}b_k=\infty$. 
 
 {\bf On H3.} Note that $f(x_{k+1}) \le C_k\le C_{k-1}\le\cdots\le C_0=f(x_0)$, the sequence $\{x_k\}_{k\in\mathbb{N}}$ generated by the NLSA is contained in the level set $\{x\in\mathbb{R}^n\,|\,f(x)\le f(x_0)\}$. Hence, whenever the level set is bounded, the sequence $\{x_k\}_{k\in\mathbb{N}}$ is bounded and there exists subsequence $\{x_{k_j}\}_{j\in\mathbb{N}}$ with $\lim\limits_{j\to\infty}x_{k_j}=\overline{x}$ such that $\lim_{j\to\infty}f(x_{k_j})=f(\overline{x})$, so condition H3 holds. Moreover, with the choices $a_k$ and $b_k$ above, we have
\[
 \overline{B} = \sup_{\mathbb{N}\ni k\ge k_1}\frac{1}{b_k}\sum_{i=k-k_1}^{k+k_1}\frac{1}{\sqrt{a_i}}
  = \sup_{\mathbb{N}\ni k\ge k_1}\frac{1}{c_1 \alpha_k}\sum_{i=k-k_1}^{k+k_1}\frac{c_2 \alpha_{i-1}}{\sqrt{\beta}}
  \le (2k_1+1) \frac{ c_2\mu}{c_1\sqrt{\beta}\underline{\alpha}}< \infty,
\]
where the inequality used the fact $\underline{\alpha}\le \alpha_k \le\mu$. 
Therefore, the assumption \eqref{abk-cond} is also satisfied.
Thus,  the NLSA of \cite{ZhangH04} with $\eta_{\max}<1$ falls within our iterative framework for $\varepsilon_k\equiv 0$ whenever the level set $\{x\in\mathbb{R}^n\,|\,f(x)\le f(x_0)\}$ is bounded.   

 Similarly, the iterate sequence $\{x^k\}_{k\in\mathbb{N}}$ generated by the nonmonotone line-search algorithms \cite{Grapiglia17,Sachs11} under the direction assumption there also falls within our iterative framework with $\varepsilon_k\equiv 0$.

\bibliographystyle{siamplain}
\bibliography{references}

\end{document}